\documentclass{amsart}
\usepackage{amssymb,amsmath,amsthm,mathtools, geometry, amssymb, verbatim, esint, soul}
\usepackage{a4wide}
\usepackage{float}
\usepackage{graphicx}
\usepackage[utf8]{inputenc} 
\usepackage{a4wide}
\usepackage{tikz}
\usetikzlibrary{calc}
\usepackage{tikz}
\usepackage{ulem}
\usepackage{mathtools}

\usepackage{stackengine,scalerel,graphicx}
\stackMath
\usepackage{hyperref}
\usepackage{amsfonts} 
\usepackage{latexsym}
\usepackage[font=small,format=hang,labelfont={sf,bf}]{caption}
\usepackage{epsfig}
\usepackage{subfig}
\usepackage{url}
\usepackage{varioref}
\usepackage{bm}
\usepackage{color}
\usepackage{mathrsfs}
\usepackage{latexsym}
\usepackage{bbm}
\usepackage{faktor}
\usepackage{yhmath}
\usepackage{empheq}
\usepackage{mathtools}
\usepackage{color}
\usepackage{marginnote}
\usetikzlibrary{matrix}
\usepackage{comment}
\usepackage{cancel}
\usetikzlibrary{arrows.meta}

\allowdisplaybreaks[1]

\newcommand{\Zd}{\mathbb Z^d}

\newcommand{\supp}{\mathrm{supp}}
\newcommand{\R}{\mathbb R}
\newcommand{\Z}{\mathbb Z}
\newcommand{\Rd}{\mathbb R^d}
\newcommand{\N}{\mathbb N}

\newcommand{\caprn}[1]{\mathrm{cap}_{N}(#1,R)}  

\newcommand{\voisin}{\mathcal{N}}
\newcommand{\capp}[1]{\mathrm{cap}\left(#1\right)}
\newcommand{\cappn}[1]{ \mathrm{cap}_{N}\left(#1\right)} 
\newcommand{\lp}{\left(}
\newcommand{\rp}{\right)}

\newcommand{\interior}[1]{\mathring{#1}}    

\def\XXint#1#2#3{{\setbox0=\hbox{$#1{#2#3}{\int}$}
     \vcenter{\hbox{$#2#3$}}\kern-.5\wd0}}

\newtheorem{theorem}{Theorem}[section]
\newtheorem*{theorem*}{Theorem}
\newtheorem{lemma}[theorem]{Lemma}
\newtheorem{proposition}[theorem]{Proposition}

\theoremstyle{definition}
\newtheorem{definition}[theorem]{Definition}
\newtheorem{remark}[theorem]{Remark}

\title{Discrete Quantitative Isocapacitary Inequality: Fluctuation Estimates}

\author[M. Cicalese]{M. Cicalese}
\address[Marco Cicalese]{Technische Universit\"at M\"unchen, Boltzmannstrasse 3, 85748 Garching, Germany	}
\email[]{cicalese@.ma.tum.de}

\author[L. Kreutz]{L. Kreutz}
\address[Leonard Kreutz]{Technische Universit\"at M\"unchen, Boltzmannstrasse 3, 85748 Garching, Germany	}
\email[]{kleo@cit.tum.de}

\author[I. Mansoor]{I. Mansoor}
\address[Imteyaz Mansoor]{DER de math\'ematiques, ENS Paris-Saclay, 91190 Gif-sur-Yvette, France}
\email[]{imteyaz.mansoor@ens-paris-sacaly.fr}

\date{}

\begin{document}

\maketitle
\begin{abstract}
The classical isocapacitary inequality states that, among all sets of fixed volume, the ball uniquely minimizes the capacity. While this result holds in the continuum, it fails in the discrete setting, where the isocapacitary problem may admit multiple minimizers. In this paper we establish quantitative fluctuation estimates for the discrete isocapacitary problem on subsets of $\Zd$
 as their cardinality diverges. Our approach relies on a careful extension of the associated variational problem from the discrete to the continuum setting, combined with sharp (continuum) quantitative isocapacitary inequalities.\end{abstract}


\tableofcontents

\section{Introduction}

Many classical problems in the calculus of variations are concerned with the identification of extremal sets for variational quantities under volume constraints. A paradigmatic example is provided by the isocapacitary inequality, which asserts that among all open and bounded sets $\Omega\subset\mathbb{R}^d$, $d\geq 3$, with prescribed Lebesgue measure, Euclidean balls uniquely minimize the Newtonian capacity. More generally, the same property holds for the capacity. More precisely, denoting by $\mathrm{cap}(\Omega)$ the capacity of $\Omega\subset\Rd$; i.e.,
\[
\mathrm{cap}(\Omega)=\inf\Bigl\{\int_{\mathbb{R}^d}|\nabla u|^2\,\mathrm{d}x \colon u\in C^\infty_c(\mathbb{R}^d)\,,\; u\geq 1 \text{ on } \Omega \Bigr\}\,,
\]
then one has 
\begin{equation*}
\mathrm{cap}(\Omega)\geq \mathrm{cap}(B_{r_{|\Omega|}})\,,
\end{equation*}
where $B_{r_{|\Omega|}}$ denotes the ball with the same volume as $\Omega$ centred at the origin. Equality holds if and only if $\Omega$ coincides with $B_{r_{|\Omega|}}$ up to translations and sets of zero $p$-capacity. 

While rigidity of minimizers is a hallmark of the continuum setting, this feature may fail in discrete environments. Understanding how continuum geometric principles manifest themselves on lattices has been a recurring theme in recent years, motivated both by intrinsic mathematical interest and by applications in probability, statistical mechanics, and materials science. \\

In this paper we study the isocapacitary problem on the integer lattice $\mathbb{Z}^d$, both in the case of the $p$-capacity and of the relative capacity, whose definitions are recalled below. Given a finite set $X\subset\mathbb{Z}^d$ with $\#X=N$, we define its discrete capacity by
\begin{equation}\label{intro:cap_discrete}
\mathrm{cap}_{N}(X)
:=\inf\Big\{N^{\frac{2-d}{d}}
\sum_{\substack{i,j\in\mathbb{Z}^d\\ |i-j|=1}}
|u(i)-u(j)|^2 \;:
u:\mathbb{Z}^d\to\mathbb{R},\ \#\mathrm{supp}(u)<\infty,\ u\ge1 \text{ on }X\Big\}\,.
\end{equation}
If moreover $X\subset B_{RN^{1/d}}$ for some $R>0$, the discrete relative capacity is defined as
\[
\mathrm{cap}_{N}(X,R)
:=\min\Big\{N^{\frac{2-d}{d}}
\sum_{\substack{i,j\in\mathbb{Z}^d\\ |i-j|=1}}
|u(i)-u(j)|^2 \;:
u= 0 \text{ on } \partial_d B_{RN^{\frac{1}{d}}}\,,\ u\ge1 \text{ on }X\Big\}\,,
\]
where $\partial_d A$ is defined in \eqref{def:discrete-boundary}.
This definition is consistent with the continuum scaling and ensures convergence, as $N\to+\infty$, to the capacity of the limiting continuum set under suitable embeddings (see Remark \ref{rem:scaling}). The latter is a consequence of the discrete-to-continuum convergence of the underlying variational problem that can be obtained in terms of $\Gamma$-convergence (we refer the reader to \cite{ABCS} for an introduction to discrete-to-contiuum variational problems). In analogy with other discrete variational problems on lattices, minimizers of $\mathrm{cap}_{N}$ (a similar observation holds for $\mathrm{cap}_{N}(\cdot,R)$) are not unique. More precisely, for fixed $N\in\N$ there may exist distinct subsets $X,Y\subset\mathbb{Z}^d$ with $\#X=\#Y=N$ minimizing $\mathrm{cap}_{N}$ and such that $X$ and $Y$ cannot be mapped onto each other by any discrete isometry of $\mathbb{Z}^d$. As customary in lattice optimization problems, this lack of rigidity naturally leads to the study of fluctuation estimates, which quantify how far different minimizers (or almost minimizers) can be from each other (for instance in terms of the cardinality of their symmetric difference).
We refer the readers to some recent results on this rigidity issue for the edge-like isoperimetric problem \cite{CL, DPS1,DPS2,  MaiPioSchSte, S}, also in the case of short-range attractive-repulsive interaction \cite{FK, FK1}, the Winterbottom problem \cite{FKS} or for the first eigenvalue of the combinatorial laplacian on $\Zd$ \cite{CiKrLeMo25}.

The main goal of this work is to establish sharp fluctuation estimates for the discrete relative and isocapacitary problem on $\mathbb{Z}^d$. Roughly speaking, we show that minimizers and almost minimizers of $\mathrm{cap}_{N}$ (or of $\mathrm{cap}_{N}(\cdot,R)$) are close, up to translations, to discrete balls of radius of order $N^{1/d}$, and that the deviation, defined in terms of the cardinality of the symmetric difference, is controlled by a power law in $N$. In the case of the capacity, in its simplest form, our main result implies that if $X$ and $Y$ are minimizers of \eqref{intro:cap_discrete}, then there exists a constant $C_{d}>0$ such that
\begin{equation}\label{intro:cap_fluct}
\#(X\Delta Y)\leq C_{d}\,N^{1-\frac{1}{2d}}\,.
\end{equation}
The latter inequality is a consequence of a more refined estimate obtained for almost minimizers of the capacity problem in Theorem~\ref{theoreme fluctuation capacite}. According to this theorem, a set $X\subset \mathbb{Z}^d$ whose $p$-capacity (a similar result holds true for the relative capacity) differs from the minimal one by $\alpha_N$, where $\sup\alpha_N<+\infty$, satisfies 
\begin{equation}\label{intro:eq-fluctuation}
    \inf_{z\in\Zd}\#(X\Delta (z+B_{r_N}\cap\Zd))\leq C_{d}N\lp\alpha_N^\frac{1}{2}+N^{-\frac{1}{2d}}P_N(X)^\frac{1}{2}\rp\,.
\end{equation}
Here $P_N(X)=N^\frac{1-d}{d}P(X)$ is a scaled version of the edge perimeter of $X$ defined as 
$$P(X)=\sum_{i\in X}\#\{j\in\Zd\setminus X\,\colon\,|i-j|=1\}\,.$$
In particular,  if $\sup P_N(X)<+\infty$, the estimate above gives
$$
\inf_{z\in\Zd}\#(X\Delta (z+B_{r_N}\cap\Zd))\leq C_{d}N\lp\alpha_N^\frac{1}{2}+N^{-\frac{1}{2d}}\rp\,,
$$
which implies \eqref{intro:cap_fluct} by the triangle inequality in the case of minimizers, that is when $\alpha_N=0$. The fact that minimizers of the discrete isocapacity problem have uniformly bounded scaled perimeter is not trivial and requires additional arguments. In Section \ref{sec:disc-rear}, using discrete rearrangement and slicing techniques, we show that minimizers of the discrete isocapacitary problem satisfy a diameter bound of order $N^{1/d}$ (Proposition \ref{prop:diameterestimate}), which in turn implies uniform control of the scaled perimeter $P_N$ (Proposition \ref{prop:perimeterestimate}). Note that \eqref{intro:cap_fluct} shows that the fluctuation estimates for exact minimizers of the discrete isocapacitary problem exhibits the same fluctuation exponent as the discrete isoperimetric problem and the discrete Faber--Krahn inequality (see \cite{CiKrLeMo25,CL}).

For what instead concern the proof of the fluctuation estimate \eqref{intro:eq-fluctuation}, our strategy closely follows the discrete-to-continuum approach first outlined for the edge isperimetric problem in \cite{CL} and is based on two main ingredients. First, we embed a discrete configuration 
$X\subset\mathbb{Z}^d$ into a continuum set 
 $\zeta(X)\subset\mathbb{R}^d$
 so that the Fraenkel asymmetry of 
$\zeta(X)$ controls the deviation of 
$X$, measured in cardinality, from a discretized Euclidean ball, with explicit error bounds in $N$. Second, we show that the discrete capacity $\mathrm{cap}_{N}(X)$ is well approximated by the continuum capacity of $\zeta(X)$, again up to explicit error bounds in $N$. The quantitative backbone of our analysis is provided by the sharp quantitative isocapacitary inequalities in the continuum setting, proved in \cite{dephilippis2019sharpquantitativeisocapacitaryinequality} for the Newtonian capacity and extended to the capacity setting in \cite{mukoseeva2023sharp}. \\

We conclude this introduction mentioning that it is our opinion that our result can be of help in the analysis of several probabilistic models with long-range correlations, such as random interlacements (see \cite{DRS}). Indeed, there the capacity naturally appears as the energetic cost governing the probability of creating macroscopic vacant regions, and hence determines the geometry of rare events and bottlenecks, for instance in first-passage percolation problems. This fact highlights the relevance of quantitative stability results for discrete isocapacitary problems in providing deterministic control on the geometry and fluctuations of near-optimal configurations.\\

The paper is organized as follows. In Section~\ref{sec:notation} we introduce the necessary notation and preliminaries. In Section~\ref{subsec:capacity} we introduce the notion of discrete relative and $p$-capacity and we state the main properties of the capacitary potentials. In Section~\ref{section:fluctuation-estimate} we state and prove the main result of the paper, namely the quantitative fluctuation results for the capacity. Finally, Section~\ref{sec:disc-rear} is devoted to establishing several geometric properties of optimal sets. These properties, proved using discrete rearrangement techniques, are required to complete the proof of the main theorem and include diameter and perimeter estimates.

\section{Notation and Preliminaries} \label{sec:notation}

For $d\geq 1$ we denote by $\{e_1, \dots, e_d\}$ the canonical basis of $\mathbb R^d$. Given $x\in\Rd$ and $r>0$, $Q_r(x)$ denotes the closed coordinate cube centered at $x$ and whose sides have length $r$. We denote by $B_r(x)$ the open ball of radius $r$ centered at $x$ and we write $B_r$ for the ball of radius $r$ centered at $0$.  Given $\alpha\geq 0,\, r_\alpha\geq 0$ denotes the radius of the ball of $d$-dimensional Lebesgue measure $\alpha$, namely $|B_{r_\alpha}|=\alpha$. Note that the symbol $| \cdot |$ will be also used to denote the Euclidian norm of an element of $\Rd$. Given a set $A$ we use the notation $\mathbbm{1}_A$ to denote the indicator function of $A$, namely the function that equals $1$ on $A$ and $0$ in its complement. The set of pairs of neighbouring points in $\Zd$ will be denoted by 
$\voisin=\big\{(i,j)\in(\Zd)^2:\, |i-j|=1\big\}.$ In the following, $C_d$ will be used to denote a positive constant depending only on the dimension that can differ from line to line. We similarly use $C_{d,R}$ (resp. $C_{d,p}$)  to denote a positive constant depending only on $d$ and $R$ (resp. $d$ and $p$). For all $p\in(1,d)$ we denote by $p^*=\frac{dp}{d-p}$ the Sobolev conjugate exponent of $p$.
\begin{definition}[\textbf{Convexity}]
    A set $X\subset \Zd$ is said to be convex in the direction $e_k$ for some $k\in \{1,\dots,d\}$ if for all $a,b\in X$ with $b=a+m e_k$ for $m\in\N$ it holds that $a+l e_k\in X$ for all $l\in\{0,\dots,m\}$. The set $X$ is said to be convex if it is convex in every direction $e_k$ for all $k\in\{1,\dots, d\}$.
\end{definition}

\subsection{Kuhn decomposition} \label{subsec:kuhn} In what follows we will make use of the Kuhn decomposition whose definition we recall below for the reader's convenience.
 The set   $[0,1]^d \subset \mathbb{R}^d$ can be decomposed in $d!$ simplexes $(T_\pi)_{\pi\in\mathcal{P}_d}$, where $\mathcal{P}_d$ is the set of permutations of $\{1,\ldots,d\}$ and for $\pi\in\mathcal{P}_d,$
    $$T_\pi=\{x\in\Rd\,\colon\,0\leq x_{\pi(d)}\leq x_{\pi(d-1)}\leq\cdots\leq x_{\pi(1)}\leq 1\}\,.$$
    Note that if $\pi,\pi'\in\mathcal{P}_d$, $\interior{T}_{\pi}\cap \interior{T}_{\pi'}=\emptyset $ if $\pi\neq\pi'$.  Given $z\in\Zd$ and $\pi\in\mathcal{P}_d$, we set  
    $$T_\pi(z)=z+T_\pi\,.$$
 The Kuhn decomposition of $\Rd$ can then be defined as 
    $$\mathcal T =\Big\{T_\pi(z)\colon z\in\Zd\,,\, \pi\in\mathcal{P}_d\Big\}\,.$$

\subsection{Discrete Dirichlet and perimeter functionals} \label{subsec:energy and perimeter} In this section, given a discrete set $X\subset\mathbb{Z}^d$ and a function $u\colon X\to\mathbb{R}$, we define several discrete functionals associated with $X$ and $u$, that will be considered in the rest of the paper. Along with them, we also define their scaled version, corresponding to the energy per particle, thus highlighting their dependence on the cardinality of $X$. Given $u\colon\Zd\to \mathbb R$, we define the discrete Dirichlet energy of $u$ as
$$E(u)=\sum_{(i,j) \in\voisin}|u(i)-u(j)|^2\,.$$
Given $X \subset \mathbb{Z}^d$ with $\#X=N$ and $u \colon \mathbb{Z}^d \to \mathbb{R}$ such that $\mathrm{supp}(u)\subset X$, we define the scaled discrete Dirichlet energy of $u$ in $X$ as
$$E_{N}(u)=N^\frac{2-d}{d}\sum_{(i,j)\in\voisin}|u(i)-u(j)|^2\,.$$

\bigskip

\noindent Given $X\subset\Zd$ with $\#X=N$, the perimeter of $X$ is defined as 
$$P(X)=\sum_{i\in X}\#\{j\in\Zd\setminus X\,\colon\,|i-j|=1\}\,,$$
and its scaled version is defined as 
$$P_N(X)=N^\frac{1-d}{d}P(X)\,.$$
Note that there exists a dimensional constant $C_d>0$ such that $C_d\leq P_N(X)\leq 2dN^\frac{1}{d}$.

\section{Discrete capacity} \label{subsec:capacity}
This section contains the main definition and basic properties of the absolute and relative discrete capacity and of their capacitary potential.\\

Let $N\in\N$ and $X\subset\Zd$ be such that $\#X=N$. The absolute capacity of $X$ is defined as 
$$\cappn{X}=\inf\Big\{E_{N}(u)\,\colon\, u\colon\Zd\to \R\,,\,\#\supp(u)<+\infty\,,\, u(i)\geq 1 \textrm{ if }i\in X\Big\}\,.$$
In order to define the relative capacity we introduce the discrete boundary of $A\subset\Rd$ as
\begin{align}\label{def:discrete-boundary}
\partial_d A := \{\{i,j\} \colon i \in A\cap \mathbb{Z}^d\,, j \in \mathbb{Z}^d \setminus A \text{ and }  |i-j|=1\}\,.
\end{align}
Given $R>0$ and $X\subset B_{RN^\frac{1}{d}}\cap\Zd$, the  relative capacity of $X$ in $B_{RN^\frac{1}{d}}\cap\Zd$  is defined as 
 $$\caprn{X}=\inf\Big\{E_N(u)\,\colon\, u\colon\Zd\to \R\,,\, u=0 \text{ on } \partial_d B_{RN^\frac{1}{d}}\,,\, u(i)\geq 1 \textrm{ if }i\in X\Big\}\,.$$ 
We set
\begin{align*}
m_{N}=\underset{\#X=N}{\inf_{X\subset \mathbb{Z}^d}} \cappn{X}\quad \text{ and } \quad  m_N(R)&=\underset{\#X=N}{\inf_{X\subset \mathbb{Z}^d}}  \caprn{X}\,.
\end{align*}
A set $X\subset\Zd$ with $\#X=N$ is called optimal, or minimizing for the  absolute isocapacitary problem if $\cappn{X}=m_{N}$ and similarly for the relative isocapacitary problem if $\caprn{X}=m_N(R)$.
\bigskip

\noindent Given a set $X \subset \mathbb{Z}^d$, a function $u \colon \mathbb{Z}^d \to \mathbb{R}$ is called the capacitary potential of $X$ (for the capacity  or the relative capacity) if it satisfies 
\begin{itemize}
\item[(1)] For the capacity:
\begin{itemize}
\item[(1.1)] $u\in \ell^{2^*}(\Zd)$; \quad (1.2)\, $u(i)=1 \quad \forall i\in X$; \quad (1.3)\, $E_N(u)=\cappn{X}$\,.
\end{itemize}
\bigskip
\item[(2)] For the relative capacity:
\begin{itemize}
\item[(2.1)] $\supp(u)\subset B_{RN^\frac{1}{d}}\cap\Zd$; \quad (2.2)\, $u(i)=1 \quad \forall i\in X$; \quad (2.3)\, $E_N(u)=\caprn{X}$\,.
\end{itemize}
\end{itemize}

\begin{remark}\label{rem:scaling} The scaling in the definition of the energy functionals $E_{N}$ is consistent with the energy functional used in the continuum definition of capacity thanks to the following observations. First, given  $u\colon\Zd\to \mathbb{R}$ and $N\in \N$, we introduce the piecewise constant interpolation of $u$ denoted by $\overline u _N\colon\Rd\to\R$ and defined as
\begin{align*}
\overline{u}_N(x)=u\lp N^\frac{1}{d}z\rp \quad \text{for }x\in Q_{N^{-\frac{1}{d}}}(z)\textrm{ and }z\in N^{-\frac{1}{d}}\Zd\,.
\end{align*}
If $\Omega \subset \mathbb{R}^d$ is open and bounded, we  set (with a slight abuse of notation)
$$E_{N}(v,\Omega)=\begin{cases}
    E_{N}(u) &\text{if } v \in L^{2^*}(\mathbb{R}^d)\,,\, v=\overline{u}_N \textrm{ with } u\colon\Zd\to\R \textrm{ and } u\geq 1 \textrm{ on }\Omega\,,\\
    +\infty &\text{otherwise.}
\end{cases}$$
We observe that as a consequence of \cite{Alicandro-Cicalese}, setting $K^2(\Rd)=\{u\in L^{2^*}(\Rd):\nabla u \in L^2(\Rd,\Rd) \}$, it holds that
$$\Gamma(L^{2^*}(\Rd))\text{-}\lim_{N\to+\infty}E_{N}(v,\Omega)=
\begin{cases}
    \displaystyle\int_{\mathbb{R}^d} |\nabla v|^2\,\mathrm{d}x&\textrm{if } v\in K^2(\Rd) \textrm{ and } v\geq 1 \textrm{ on } \Omega\,,\\
    +\infty &\textrm{ otherwise,}
\end{cases}$$
Therefore, thanks to the coercivity properties of $E_{N}$ and the fundamental theorem of $\Gamma$-convergence \cite{Braides, DalMaso}, for any $\Omega\subset\Rd$ open and $\Omega_N= N^{\frac{1}{d}} \Omega \cap \mathbb{Z}^d$, we have $\mathrm{cap}_{N}(\Omega_N) \to \mathrm{cap}(\Omega)$ as $N\to\infty$, where 
\begin{align*}
\mathrm{cap}(\Omega) = \min\left\{\int_{\mathbb{R}^d} |\nabla u|^2\,\mathrm{d}x \colon u \in K^2(\mathbb{R}^d)\,, u \geq 1 \text{ on } \Omega\right\}\,.
\end{align*}
A similar result holds in the case of the relative capacity, where the additional requirement $\mathrm{supp}(u) \subset B_{RN^{\frac{1}{d}}}$ implies that $\mathrm{cap}_N(\Omega,R) \to \mathrm{cap}(\Omega,R)$ as $N\to \infty$, where
\begin{align*}
\mathrm{cap}(\Omega,R) = \min\left\{\int_{\mathbb{R}^d} |\nabla u|^2\,\mathrm{d}x \colon u \in H^1_0(B_R)\,, u \geq 1 \text{ on } \Omega\right\}\,.
\end{align*}
\noindent Note also that if $P_N$ is extended to $L^1(\Rd)$ as

$$P_N(v)= \begin{cases}
    P_N(X) \,\textrm{ if } v=\mathbbm{1}_X \textrm{ with } \#X=N\,,\\
    +\infty  \quad\,\, \textrm{ else},
\end{cases}$$
then, as stated in \cite{Alicandro-Braides-Cicalese-05, CiKrLeMo25}, it holds that
$$\Gamma(L^{1}(\Rd))~\text{-}\lim_{N\to+\infty}P_N(v)=
\begin{cases}
   \displaystyle \int_{\partial^*E}\|\nu_E\|_1 \mathrm{d}\mathcal{H}^{d-1} \,\,\,\textrm{if } v=\mathbbm{1}_E\in BV(\Rd), \|v\|_1=1\,,\\
    +\infty \quad\quad\quad\quad\quad\quad \textrm{ otherwise,}
\end{cases}$$
where $\partial^*E$ denotes the reduced boundary of $E$, and $\nu_E$ is its measure theoretic unit normal.
\end{remark}

\subsection{Properties of the capacitary potential}\label{sec:potential} \noindent Here we provide the main properties of the capacitary potential useful in the sequel of the paper.

\begin{theorem}{\cite[Theorem 1]{discreteGNS}} \label{thm:GNS}
    Let $p\in [1,d)$ and $u \in\ell^p(\Zd)$.
 There exists a constant $C_{p,d}>0$ such that
    \begin{equation*}
    \lp\sum_{i\in \Zd}|u(i)|^{p^*}\rp^\frac{1}{p^*}\leq C_{p,d} \lp\sum_{(i,j)\in \voisin}|u(i)-u(j)|^{p}\rp^\frac{1}{p}\,.
    \end{equation*}
\end{theorem}

\begin{definition}  \label{definition K2}
Let $X\subset\Zd$ such that $\#X=N$.
\begin{enumerate}
    \item Let $K^2(\Zd)$ be the vector space 
    $$K^2(\Zd)=\Big\{u\in\ell^{2^*}(\Zd)\colon E(u)<+\infty\Big\}\,,$$
    the set of functions in $\ell^{2^*}(\Zd)$ with finite Dirichlet energy.   Equipped with the norm $|||u|||_{K^2}=\|u\|_{\ell^{2^*}(\Zd)}+E(u)^\frac{1}{2}$, the set $K^2(\Zd)$ is a Banach space. Thanks to Theorem~\ref{thm:GNS} the latter norm is equivalent to the norm $\|u\|_{K^2}=E(u)^{\frac{1}{2}}$. The space $K^2(\Zd)$ endowed with such a norm is Hilbert space.
    \item We observe that the set
        $$C_X=\{u\in K^2(\Zd)\colon u(i)=1 \text{ if } i\in X \}\,,$$
     of all admissible functions in the definition of $\cappn{X}$, is the closure (w.r.t.~${\|\cdot \|_{K^2}}$) of
    $$C_X^{comp}=\{u\in K^2(\Zd)\colon\ u(i)=1 \text{ if } i\in X \textrm{ and } \#\supp(u)<+\infty\}\,.$$ 
\end{enumerate}   
\end{definition}

\begin{proposition}\label{prop:cap_pot}
    Given $X\subset \Zd$ and $\#X=N$. Then the following holds true. 
        \begin{itemize}
           \item[(1)] If $X$ is such that $\caprn{X}<+\infty$.  Then, there exists a unique capacitary potential $u \in K^2(\mathbb{Z}^d)$ such that $\mathrm{supp}(u) \subset B_{RN^{\frac{1}{d}}}$;
    \item[(2)] If $X$ is such that $\cappn{X}<+\infty$.  Then, there exists a unique capacitary potential $u \in K^2(\mathbb{Z}^d)$.
    \end{itemize}
\end{proposition}
\begin{proof} First, we note the following convexity property of $E_{N}$. Given $u,v\colon \Zd\to \R$ admissible and $t\in (0,1):$
    \begin{align}\label{ineq:convexityEpN}
    \begin{split}
        E_{N}(tu+(1-t)v)=&\sum_{\substack{i,j \in\Zd\\|i-j|=1}}\big|t(u(i)-u(j))+(1-t)(v(i)-v(j))|^2\\
        \leq&~t \sum_{\substack{i,j \in\Zd\\|i-j|=1}}|u(i)-u(j)|^2+\,\,\,(1-t) \sum_{\substack{i,j \in\Zd\\|i-j|=1}} |v(i)-v(j)|^2\,,
        \end{split}
    \end{align}
    with equality only if $ u(i)-u(j)=v(i)-v(j) \quad \forall \, (i,j)\in \voisin$. Now we prove {\rm (1)}.  Since $u\equiv 0 $ on $\mathbb{Z}^d \setminus B_{RN^{\frac{1}{d}}}$,
 the existence of the capacitary potential follows as the minimum problem defining the relative capacity reduces to minimizing a continuous function over a finite dimensional set. Concerning uniqueness, we note that the set $$\Big\{u\colon\Zd\to \R\colon u\geq 1 \text{ on }X \text{ and }\supp(u)\subset B_{RN^{\frac{1}{d}}}\Big\}$$ is convex. The characterization of the equality case in \eqref{ineq:convexityEpN} together with the fact that  $v=u= 0 $ on $\mathbb{Z}^d \setminus B_{RN^{\frac{1}{d}}}$ and $\mathbb{Z}^d$ is path-connected through paths of nearest-neighbors, implies that $E_N$ is strictly convex.
    
    \bigskip
    
  \noindent  Next, we prove {\rm (2)}.  The existence of the capacitary potential follows from the reflexivity of $K^2$ and the lower semicontinuity of $E_{N}$. Again, the uniqueness follows from the strict convexity of $E_{N}$ on $C_X$ due to \eqref{ineq:convexityEpN} and the fact that $u, v \in \ell^{2^*}(\mathbb{Z}^d)$.
\end{proof}

\begin{definition} We say that a function $u \colon \mathbb{Z}^d \to \mathbb{R}$ is harmonic at $i \in \mathbb{Z}^d$ if 
        \begin{align*}
        \sum_{|j-i|=1}(u(i)-u(j))=0\,.
\end{align*}     
\end{definition}

In the following proposition we collect some well known necessary conditions for minimality (see~\cite{dephilippis2019sharpquantitativeisocapacitaryinequality} in the continuum case).
\begin{proposition} \label{prop:function energie min}
    Let $X\subset \Zd$ and $\#X=N$ and $u\colon\Zd\to \R$ its capacitary potential.
    \begin{enumerate}
        \item $\forall i \in \Zd: 0\leq u(i)\leq 1$;
        \item $\forall i \in X:u(i)=1$;
        \item[(3.1)] If $E_N(u)=\caprn{X}$ and if $i \in B_{RN^\frac{1}{d}} \cap \mathbb{Z}^d \setminus (X \cup \partial_d B_{RN^\frac{1}{d}})$, then $u$ is harmonic at $i$. In particular $u(i)>0$ for all $i \in B_{RN^{\frac{1}{d}}} \cap \mathbb{Z}^d$.
        \item[(3.2)] If $E_N(u)=\cappn{X}$ and if $i \in  \mathbb{Z}^d \setminus X$, then $u$ is harmonic at $i$.
    \end{enumerate}
\end{proposition}
\begin{proof} Property {\rm (1)} is a consequence of the fact that $E_{N}$ is decreasing under truncation and, if $u \in C_X$, then $\hat{u} = (u \wedge 0)  \vee 1 \in C_X$. Property {\rm (2)}  is a direct consequence of {\rm (1)}, and the fact that all competitors $u$ satisfy $u(i)\geq1$ on $X$.  For {\rm (3.1)} and {\rm (3.2)}, take $i\notin X$ (and $i\in B_{RN^\frac{1}{d}}$ in the case of the relative capacity), and $v=\mathbbm{1}_{\{i\}}$. Since $u+tv$ is also admissible for every $t\in \R$, the first-order condition shows that it is harmonic. In particular in the case of the relative capacity, it follows that $u(i) >0$ for all $i \in B_{RN^{\frac{1}{d}}} \cap \mathbb{Z}^d$.
\end{proof}

\section{Fluctuation estimates} \label{section:fluctuation-estimate}

In this section we establish the main result of the paper, which is stated in the theorem below. The general strategy of the proof follows the same approach developed in \cite{CiKrLeMo25}.
\begin{theorem} \label{theoreme fluctuation capacite}
    Let $\{\alpha_N\}_N \subset (0,+\infty)$ be such that $\sup\alpha_N<+\infty$. 
    
 \begin{itemize}
 \item[(i)] capacity: let $X\subset\Zd$ satisfy $\#X=N$ and
 \begin{equation} \label{ineq:alphamin-p}
  \cappn{X}\leq m_{N}+ \alpha_N\,.
 \end{equation}
Then, there exists a constant
 $C_{d}>0$ such that 
    \begin{equation} \label{eq-fluctuation}
    \inf_{z\in\Zd}\#(X\Delta (z+B_{r_N}\cap\Zd))\leq C_{d}N\lp\alpha_N^\frac{1}{2}+N^{-\frac{1}{2d}}P_N(X)^\frac{1}{2}\rp.
    \end{equation}
       In particular,  if $\sup P_N(X)<+\infty$, then
    $$\inf_{z\in\Zd}\#(X\Delta (z+B_{r_N}\cap\Zd))\leq C_{d}N\lp\alpha_N^\frac{1}{2}+N^{-\frac{1}{2d}}\rp\,.$$
     \item[(ii)] relative capacity: let $X\subset\Zd$ satisfy $\#X=N$ and
 \begin{equation} \label{ineq:alphamin-R}
 \caprn{X}\leq m_{N}(R)+ \alpha_N\,.
 \end{equation}
 Then, there exists a constant  $C_{d,R}>0$ such that
    \begin{equation} \label{eq fluctuation relative}
    \#(X\Delta (B_{r_N}\cap\Zd))\leq C_{d,R}N\lp\alpha_N^\frac{1}{2}+N^{-\frac{1}{2d}}P_N(X)^\frac{1}{2}\rp\,.
    \end{equation}
\end{itemize}    
\end{theorem}

\begin{remark}
In Section \ref{sec:diameter estimate}, we will show that the condition $\sup P_N(X)<+\infty$ is satisfied for sets $X$ such that $\# X=N$ and $  \cappn{X}=m_{N}$ or   $\caprn{X}=m_N(R)$.\\
\end{remark}

\begin{definition} \label{continuous embedding}
    Let $X\subset \Zd$ and $u:\Zd\to \R$ be its discrete capacitary potential. We define $\zeta(X)\subset\Rd$ as
    $$\zeta(X)=\mathrm{int}\left(\bigcup_{\substack{T\in\mathcal{T}\\T\cap \Zd\subset X}}T \right),$$
    and we denote by $\hat u$ the affine interpolation of $u$ on the simplices of the Kuhn decomposition.  
\end{definition}
Note that, in the case of the relative capacity, for $X \subset B_{RN^\frac{1}{d}} \cap \mathbb{Z}^d$ with $\mathrm{cap}_{R,N}(X) <+\infty$ we have that $\zeta(X) \subset B_{RN^{\frac{1}{d}}}$.  

\begin{lemma} \label{lemme estimation mesure zeta}
    Let $X\subset \Zd$ be such that $\#X=N$. Then there exists $C_d>0$ such that 
    \begin{equation} \label{eq estimation mesure zeta}
    N-C_dN^\frac{d-1}{d}P_N(X)\leq |\zeta(X)|\leq N \,.
    \end{equation}
\end{lemma}
\begin{proof}
    Let
    $$ X' =\big\{i\in X \,\colon\, j \in X \text{ for all } j \in \mathbb{Z}^d \text{ such that } |i-j|\leq \sqrt{d}\big\}.$$
For all $i \in X'$ we have that $Q_1(i)\subset \zeta(X)$ and therefore
    $$\# X'=\Big|\bigcup_{i\in X'}Q_1(i)\Big|\leq |\zeta(X)|.$$
  By definition of $X'$ we have that
    $$\# X'\geq \#X- C_dP(X)=N-C_dN^\frac{d-1}{d}P_N(X)\,.$$  
    The second inequality comes from the fact that for every $T\in\mathcal{T}$ satisfying,
    $T\cap\Zd\subset X,$
    $$T\subset \bigcup_{i\in T\cap\Zd}Q_1(i)\subset\bigcup_{i\in X}Q_1(i)\,.$$
\end{proof}

\begin{lemma}[Lemma 4.5, \cite{CiKrLeMo25}] \label{lemme estimation gradient}
Given $u:\Zd\to \R$, its piecewise-affine interpolation $\hat u$ satisfies
$$\int_{\Rd} |\nabla \hat u(x)|^2\, \mathrm{d}x=N^\frac{d-2}{d}E_{N}(u)\,.$$ 
\end{lemma}

The next lemma follows from our definitions of capacities in Subsection~\ref{subsec:capacity} and Lemma~\ref{lemme estimation gradient}.
\begin{lemma}  \label{lemme comparaison continu discret}
    Let $X\subset \Zd$ be such that $\#X=N$. Then in the case of capacity,
     $$\capp{\zeta(X)}\leq N^{\frac{d-2}{d}}\cappn{X}\,.$$
    In the case of relative capacity for $R>0$, assuming $X\subset B_{RN^\frac{1}{d}}\cap \Zd$, we have $\zeta(X) \subset B_{RN^{\frac{1}{d}}}$ and
    $$\mathrm{cap}\left(\zeta(X),RN^{\frac{1}{d}}\right)\leq  N^\frac{d-2}{d}\caprn{X}\,.$$

\end{lemma}

We recall the following characterization of radial harmonic functions in $\mathbb{R}^d\setminus\{0\}$ (see \cite{flucher}). 

\begin{proposition} \label{prop:harmonic} Let $d\geq 3$. The capacitary potential  and the capacity of $B_1$ are given by
  \begin{align}\label{propeq:p-capacitypotential}
    u_2(x)=\begin{cases}
    1&\text{if } |x| \leq 1\,,\\
   |x|^{2-d}  &\text{otherwise}
\end{cases} \quad \text{ and } \quad \mathrm{cap}(B_1)= d(d-2)|B_1|\,.
    \end{align}

For $R>1$ the relative capacitary potential and the relative capacity of $B_1$ are given by
    \begin{equation} \label{propeq:relative-capacitypotential}
 u_R(x) =\begin{cases}
    1&\text{if } |x|\leq 1\,,\\
   \frac{1}{1-R^{2-d}}\left(|x|^{2-d}-R^{2-d}\right) &\text{otherwise}
\end{cases} \quad\text{ and } \quad  \mathrm{cap}(B_1,R)=\frac{|B_1|d(d-2)}{\lp1-R^{2-d}\rp}\,.
    \end{equation}
\end{proposition}

\begin{lemma} \label{lemme borne capacite min}
    Let $N\in \mathbb N$, and $R>2$. There exists $C_d>0$ such that 
    $$m_{N}\leq |B_1|^\frac{2-d}{d}\capp{B_1}+C_{d}N^{-\frac{1}{d}}$$
    for the capacity, and 
    $$m_N(R)\leq |B_1|^\frac{2-d}{d} \mathrm{cap}(B_1,R|B_1|^{\frac{1}{d}})+C_{d}N^{-\frac{1}{d}}$$
    for the relative capacity.
\end{lemma}
\begin{proof}
We only prove the result for the capacity. For the time being, we assume that $N=N_k=\#(\overline{B_k}\cap \Zd)$ for $k>0$ (the general case will be considered later). Let $C_d>0$ be such that
\begin{align}\label{ineq:cardball}
|B_1|\lp k^d-C_dk^{d-1}\rp\leq N\leq|B_1|\lp k^d+C_dk^{d-1}\rp\,.
\end{align}
Let $u_k\colon\Zd\to\R$ be defined as $u_k(i)=u_2(\frac{i}{k})$ for $i\in \Zd$, where $u_2$ is given in formula \eqref{propeq:p-capacitypotential}. Note that $u_k$ is admissible for $\mathrm{cap}(\overline{B}_k\cap \mathbb{Z}^d) $, hence
\begin{align*}
m_{N} \leq \mathrm{cap}(\overline{B}_k\cap \mathbb{Z}^d) \leq E_{N}(u_k)\,.
\end{align*}
Therefore, to obtain the statement of the lemma it suffices to show that
\begin{align} \label{ineq:ukcapacity-p}
E_{N}(u_k) \leq |B_1|^\frac{2-d}{d}\capp{B_1}+C_{d}N^{-\frac{1}{d}}\,.
\end{align}
To this end, we note that, by the regularity of $u_2$ and the local Lipschitzianity of $x \mapsto x^2$ there exists $C_d>0$ such that
\begin{align*}
|u_k(i+e_n)-u_k(i)|^2 \leq k^{d-2} \int_{Q_{\frac{1}{k}}\left(\frac{i}{k}\right)} |\partial_n u_2(x)|^2\,\mathrm{d}x + C_{d} k^{-3}\|\partial_n u_2\|_{L^\infty(Q_{\frac{1}{k}}(\frac{i}{k}))} \|D^2u\|_{L^\infty(Q_{\frac{1}{k}}(\frac{i}{k}))}  \,.
\end{align*}
We note that by \eqref{propeq:p-capacitypotential} we have
\begin{align*}
\sum_{i \in \mathbb{Z}^d \setminus \overline{B}_k}\|\partial_n u_2\|_{L^\infty(Q_{\frac{1}{k}}(\frac{i}{k}))} \|D^2u\|_{L^\infty(Q_{\frac{1}{k}}(\frac{i}{k}))} \leq C_{d} k^{2d-1} \sum_{i \in \mathbb{Z}^d \setminus \overline{B}_k} |i|^{-2d+1} \leq  C_{d} k^{d}\,.
\end{align*}
Finally, noting that $u_k$ is constant in $\overline{B}_k \cap \mathbb{Z}^d$, summing over $i \in \mathbb{Z}^d$ and $n \in \{1,\ldots,d\}$, thanks to \eqref{ineq:cardball} we obtain
\begin{align*}
E_{N}(u_k) &\leq N^{\frac{2-d}{d}} k^{d-2} \int_{\mathbb{R}^d} |\nabla u_2|^2\,\mathrm{d}x +C_{d} N^{\frac{2-d}{d}} k^{d-3} \\&\leq  |B_1|^{\frac{2-d}{d}}\int_{\mathbb{R}^d} |\nabla u_2|^2\,\mathrm{d}x +C_{d} N^{-\frac{1}{d}} =  |B_1|^{\frac{2-d}{d}}\mathrm{cap}(B_1)+C_{d} N^{-\frac{1}{d}} \,.
\end{align*}
For general $N\in \mathbb{N}$ we observe that there exists $k>0$ such that, setting $N_k =\#(\overline{B}_k \cap \mathbb{Z}^d)$, one has $0 \leq N_k-N\leq C_d N^{\frac{d-1}{d}}$. Finally, the estimate follows by using a test function $u_k$ constructed as above (suitably scaled). In the case of the relative capacity the proof follows along the same line by discretizing the function $u_{R-\sqrt{d}N^{-\frac{1}{d}}}$.
\end{proof}

Next, we state an auxiliary Lemma, whose key feature lies in the fact that the parameter $\delta>0$ can be chosen (locally) uniformly with respect to $R$.

\begin{lemma} \label{lem:no-concentration-bdry} Let $\eta>0$, $R_0 >1$, and $\sigma>0$. Let $\bar R \in [R_0,2R_0]$ and let $\Omega \subset B_R$ such that $|\Omega|=|B_1|$ with
\begin{align}\label{lem-ineq:almost-min}
\mathrm{cap}(\Omega,\bar R) \leq \mathrm{cap}(B_1,\bar R) +\eta\,.
\end{align} 
There exists $C=C(d,R_0,\eta)$ such that for $\delta= C\cdot\sigma>0$ there holds
\begin{align}\label{lem-ineq:measure}
\left| \Omega\setminus B_{(1-\delta)\bar R} \right| \leq \sigma\,.
\end{align}
\end{lemma}
\begin{proof} 
Assume by contradiction that there exists $\Omega \subset B_{\bar R}$ and $\sigma>0$ such that for all $C>0$ with $\delta=C\cdot\sigma<1$ it holds
\begin{align*}
\left| \Omega\setminus B_{(1-\delta)\bar R} \right| > \sigma\,.
\end{align*}
 Given $0<r<\bar R$ we define  
\begin{align}\label{set:E-r}
\Omega_{r} = \{\xi \in \partial B_1 \colon r\xi \in \Omega \} \,.
\end{align}
We prove that there exists $r_0 \in ((1-\delta)\bar R,\bar R)$ such that
\begin{align}\label{ineq:measure-sphere}
\mathcal{H}^{d-1}(\partial B_1 \cap \Omega_{r_0})> \frac{\sigma}{\delta \bar R} r_0^{1-d}\,.
\end{align}
Indeed, by the Coarea formula and the mean value theorem,  there exists $r_0$ as above such that
\begin{align*}
\sigma <  |\Omega\cap B_{\bar R} \setminus B_{(1-\delta)\bar R}| = \int_{(1-\delta)\bar R}^{\bar R} \mathcal{H}^{d-1}(\partial B_r \cap \Omega)\,\mathrm{d}r \leq \delta \bar R\, \mathcal{H}^{d-1}(\partial B_{r_0} \cap \Omega) = \delta \bar R r_0^{d-1} \mathcal{H}^{d-1}(\Omega_{r_0})\,,
\end{align*}
where we used that $\xi \in \Omega_{r_0}$ if and only if $r_0 \xi \in \Omega\cap \partial B_{r_0}$ and the $(d-1)$-homogeneity of $\mathcal{H}^{d-1}$. This shows \eqref{ineq:measure-sphere}. Let $u$ be the relative capacitary potential of $\Omega$.  We claim that 
\begin{align}\label{ineq:Dirichlet-measure-ineq}
\int_{B_{\bar R}} |\nabla u|^2\,\mathrm{d}x \geq (d-2) \mathcal{H}^{d-1}(\Omega_{r_0}) \left(r_0^{2-d}-\bar R^{2-d} \right)^{-1}\,.
\end{align}
Assuming the claim, we show how this, together with \eqref{ineq:measure-sphere}, implies the statement of the Lemma. In fact, using \eqref{propeq:relative-capacitypotential}, \eqref{ineq:measure-sphere} and \eqref{ineq:Dirichlet-measure-ineq},  we obtain
\begin{align*}
C_d \frac{\sigma}{\delta}\bar R^{-2}\leq  \frac{\sigma}{\delta \bar R} r_0^{1-d}\left(r_0^{2-d}-\bar R^{2-d} \right)^{-1}  &< (d-2) \mathcal{H}^{d-1}(\Omega_{r_0}) \left(r_0^{2-d}-\bar R^{2-d} \right)^{-1} \leq\int_{B_{\bar R}} |\nabla u|^2\,\mathrm{d}x \\&\leq \mathrm{cap}(B_1,\bar R) +\eta \leq \frac{|B_1|d(d-2)}{\lp1-\bar R^{2-d}\rp} +\eta\,.
\end{align*}
This leads to a contradiction for $\delta=C\cdot \sigma$ with 
\begin{align*}
C \leq C_d\left (\frac{|B_1|d(d-2)}{\lp1-\bar R^{2-d}\rp} +\eta\right)^{-1} \bar R^{-2}\leq C(d,R_0,\eta)\,.
\end{align*}
We are left to prove the claim \eqref{ineq:Dirichlet-measure-ineq}. Again by the Coarea formula we obtain
\begin{align*}
\int_{B_{\bar R}} |\nabla u|^2\,\mathrm{d}x &= \int_0^{\bar R} \int_{\partial B_r} |\nabla u|^2\,\mathrm{d}\mathcal{H}^{d-1}\,\mathrm{d}r = \int_0^{\bar R} \int_{\partial B_1} r^{d-1} |\nabla u(r\xi )|^2\,\mathrm{d}\mathcal{H}^{d-1} 
\,\mathrm{d}r \\& \geq \int_{r_0}^{\bar R} \int_{ \Omega_{r_0}} r^{d-1} |\nabla u(r\xi )|^2\,\mathrm{d}\mathcal{H}^{d-1} 
\,\mathrm{d}r = \int_{ \Omega_{r_0}}  \int_{r_0}^{\bar R} r^{d-1} |\nabla u(r\xi )|^2\,\mathrm{d}r \,\mathrm{d}\mathcal{H}^{d-1} \,.
\end{align*}
Now, \eqref{ineq:Dirichlet-measure-ineq} follows provided we show for $\mathcal{H}^{d-1}$-a.e.~$\xi \in \Omega_{r_0}$ there holds 
\begin{align}\label{ineq:line-estimate}
\int_{r_0}^{\bar R} r^{d-1} |\nabla u(r\xi )|^2\,\mathrm{d}r  \geq (d-2)  \left(r_0^{2-d}-\bar R^{2-d} \right)^{-1}\,.
\end{align}
To this end, given $\xi\in\Omega_{r_0}$, we define $w(r) = u(r\xi)$. Then $w \in W^{1,2}((r_0,\bar R))$ with $w'(r)= \nabla u(r\xi) \cdot \xi$, $w(r_0)=1$, and  $w(\bar R)=0$. Then, as $\xi \in\partial B_1$ we have
\begin{align*}
\int_{r_0}^{\bar R} r^{d-1} |\nabla u(r\xi )|^2\,\mathrm{d}r  \geq \int_{r_0}^{\bar R} r^{d-1} |\nabla u(r\xi )\cdot \xi|^2\,\mathrm{d}r = \int_{r_0}^{\bar R} r^{d-1} |w'(r)|^2\,\mathrm{d}r \geq (d-2)  \left(r_0^{2-d}-{\bar R}^{2-d} \right)^{-1}\,,
\end{align*}
where in the last inequality we explicitly minimized the integral among $w \in W^{1,2}((r_0,{\bar R}))$ with the aforementioned boundary conditions. This shows \eqref{ineq:line-estimate} and concludes the proof of this Lemma.
\end{proof}
\begin{lemma} \label{lem:diffiomorphism} Let $0<\delta,\varepsilon<1$ and let $0<\sigma<\frac{1}{2}$. Let $\Omega \subset B_R$ be such that $|\Omega|=|B_1| -\varepsilon$ and assume that $|\Omega\setminus B_{(1-\delta)R}| \leq \sigma$. Then there exists a bi-Lipschitz function $\Phi_{\varepsilon,\delta,\sigma} \colon B_R \to B_R$ such that 
\begin{align}\label{lem-ineq:diff}
\|\nabla \Phi_{\varepsilon,\delta,\sigma} -\mathrm{Id}\|_{L^\infty(B_R)} \leq C_d \varepsilon\,\sigma\, \delta^{-1}\,, \quad \text{ and } \quad  |\Phi_{\varepsilon,\delta,\sigma}(\Omega)| = |B_1|\,.
\end{align}
In particular,
\begin{align}\label{lem-ineq:diff-cap}
\mathrm{cap}(\Phi_{\varepsilon,\delta,\sigma}(\Omega),R) \leq \mathrm{cap}(\Omega,R) +C_d\,\varepsilon\, \sigma \,\delta^{-1}\,\,.
\end{align}
\end{lemma}
\begin{proof} For $\lambda \geq 0$ we define $\Phi_\lambda \colon B_R \to B_R$ by
\begin{align}\label{def:Phi-lambda}
\Phi_\lambda(x) = \begin{cases} (1+\lambda) x &\text{if } x \in B_{(1-\delta)R} \,,\\
x(1+\lambda\frac{R-|x|}{\delta R})&\text{if } x \in B_{R}\setminus  B_{(1-\delta)R}\,.
\end{cases} 
\end{align}
The first part of \eqref{lem-ineq:diff} is elementary provided $\lambda \leq C_d\varepsilon \sigma \delta^{-1}$.
We want to prove that there exists $\bar\lambda \leq C_d \varepsilon \sigma \delta^{-1}$ such that $\Phi_{\bar\lambda}(\Omega)=|B_1|$. This follows from the continuity of $\lambda \mapsto |\Phi_\lambda(\Omega)|$, observing that $|\Phi_0(\Omega)| = |B_1|-\varepsilon$ and that $|\Phi_{\lambda_0}(\Omega)| \geq |B_1|$ for $\lambda_0=C_d \sigma \varepsilon \delta^{-1}$. In fact we have that
\begin{align*}
|\Phi_\lambda(\Omega)| &= \int_{\Omega} |\det(\nabla \Phi_\lambda(x))|\,\mathrm{d}x = (1+\lambda)^d|\Omega| + \int_{\Omega \setminus B_{(1-\delta)R}} |\det(\nabla \Phi_\lambda(x))|-(1+\lambda)^d\,\mathrm{d}x\\&\geq (1+\lambda)^d(|B_1|-\varepsilon) -C_d\sigma  \varepsilon \,\delta^{-1} \geq |B_1|
\end{align*}
provided that $\lambda \geq C_d \sigma \varepsilon \delta^{-1}=\lambda_0$. Clearly \eqref{lem-ineq:diff-cap} follows from the first part of \eqref{lem-ineq:diff}, using as a competitor for $\Phi_{\varepsilon,\delta,\sigma}(\Omega)$ the function $u \circ \Phi^{-1}_{\varepsilon,\delta,\sigma}$, where $u$ is the capacitary potential of $\Omega$, and a changing variables. Finally we set $\Phi_{\varepsilon,\delta,\sigma} = \Phi_{\bar\lambda}$.
\end{proof}

\begin{lemma} \label{lemme comparaison diff symetrique discrete} There exists $C_d>0$ such that for all $N \in \mathbb{N}$ and $X\subset \Zd$ with $\#X=N$, and for all $z\in\Zd$ it holds  
    $$\#\lp X\Delta(z+B_{r_N}\cap\Zd)\rp\leq \big|\zeta(X)\Delta(z+B_{r_N})\big|+C_dN^\frac{d-1}{d}P_N(X)
    \,.$$
\end{lemma}

\begin{proof}
    The proof follows along the same lines as the proof  of Lemma 4.8 in \cite{CiKrLeMo25}.
\end{proof}

It is now possible to prove the fluctuation estimate for the discrete isocapacitary inequality, in a similar way to Theorem 4.1 of \cite{CiKrLeMo25}.

\begin{proof}[Proof of Theorem~\ref{theoreme fluctuation capacite}] We divide the proof into two steps. We first prove the result for the $p$-capacity and then for the relative capacity. In what follows we can assume 
\begin{align}\label{eq:assumptions-for-proof}
\zeta(X) \neq \emptyset \quad \text{ and } \quad \limsup_{N\to +\infty}  N^{-\frac{1}{d}}P_N(X) =0\,.
\end{align}
 as the statement follows trivially otherwise. 

\noindent {\bf Step 1:} {\it capacity.} Let $N\in\N$, and let $X\subset\Zd$ be such that $\#X=N$, and  \eqref{ineq:alphamin-p} holds, i.e.,
\begin{align}\label{ineq:almostmin-p-prove}
\mathrm{cap}_{N}(X) \leq m_{N} + \alpha_N\,.
\end{align}
By Lemma~\ref{lemme comparaison continu discret}, Lemma~\ref{lemme borne capacite min} and \eqref{ineq:almostmin-p-prove} we obtain
 \begin{align}\label{ineq:capzetaX}
 \capp{\zeta(X)}\leq \left(\frac{N}{|B_1|}\right)^{\frac{d-2}{d}} \left(\capp{B_1} + C_d(\alpha_N +N^{-\frac{1}{d}})\right)\,.
 \end{align}
We set $r=r_{|\zeta(X)|}=\lp\frac{|\zeta(X)|}{|B_1|}\rp^\frac{1}{d}$ and we use the scaling properties of the capacity together with Lemma~\ref{lemme estimation mesure zeta}, \eqref{eq:assumptions-for-proof}, and $(1-x)^{\frac{d-2}{d}} \geq  1-\frac{d-2}{2d}\,x$ for $x\geq 0$ small enough, to obtain
$$\capp{B_r}=r^{d-2}\capp{B_1}=\lp\frac{|\zeta(X)|}{|B_1|}\rp^\frac{d-2}{d}\capp{B_1}\geq \left(\frac{N}{|B_1|}\right)^\frac{d-2}{d}\lp1-C_{d}N^{-\frac{1}{d}}P_N(X)\rp\capp{B_1}\,.$$
The latter estimate together with \eqref{eq:assumptions-for-proof}, \eqref{ineq:capzetaX}, and recalling that $P_N(X) \geq C_d$, yields
\begin{align*}
 \frac{\capp{\zeta(X)}-\capp{B_r}}{\capp{B_r}} \leq C_{d}\left(\alpha_N+N^{-\frac{1}{d}}P_N(X)  \right)\,.
\end{align*}
Using the subadditivity of the square root,  we obtain
\begin{align*}
\lp\frac{\capp{\zeta(X)}-\capp{B_r}}{\capp{B_r}}\rp^\frac{1}{2}\leq C_{d}\lp\alpha_N^\frac{1}{2}+N^{-\frac{1}{2d}}P_N(X)^\frac{1}{2}\rp\,.
\end{align*}
By \cite[Theorem 1.4]{mukoseeva2023sharp} and Lemma~\ref{lemme estimation mesure zeta} there exists $z\in \mathbb{R}^d$
\begin{align}\label{ineq:zetaBr}
|\zeta(X)\Delta (z+B_r)| \leq  C_{d} N \lp\alpha_N^\frac{1}{2}+N^{-\frac{1}{2d}}P_N(X)^\frac{1}{2}\rp\,.
\end{align}
Let $z'\in\Zd$ be such that $|z-z'|\leq \sqrt d$. By the triangle inequality we obtain 
\begin{align}\label{ineq:fourthestimate41}
|\zeta(X)\Delta(z'+ B_{r_N})|\leq  |\zeta(X)\Delta(z+ B_{r})|+|(z+ B_{r})\Delta(z'+ B_{r})|+|(z'+ B_{r})\Delta(z'+ B_{r_N})|\,.
\end{align}
Recalling the definition of $r$, using Lemma~\ref{lemme estimation mesure zeta}, and the fact that $P_N(X) \geq C_d$, we obtain
\begin{align}\label{ineq:zetaprimoBr} 
    |(z+ B_{r})\Delta(z'+ B_{r})|& \leq C_d |\zeta(X)|^\frac{d-1}{d}\leq C_dN^\frac{d-1}{d}P_N(X)\,.
\end{align}
According to Lemma~\ref{lemme estimation mesure zeta} we obtain 
\begin{align}\label{ineq:ineqprimoBrN}
|(z'+ B_{r})\Delta(z'+ B_{r_N})|=|B_{r_N}\setminus B_r|\leq C_dN^{\frac{d-1}{d}}P_N(X)\,.
\end{align}
Thanks to \eqref{eq:assumptions-for-proof} and \eqref{ineq:zetaBr}--\eqref{ineq:ineqprimoBrN}  we obtain
\begin{align*}
|\zeta(X)\Delta(z'+ B_{r_N})| \leq C_{d}N\lp\alpha_N^\frac{1}{2}+N^{-\frac{1}{2d}}P_N(X)^\frac{1}{2}\rp+C_dN^\frac{d-1}{d}P_N(X) \leq C_{d}N\lp\alpha_N^\frac{1}{2}+N^{-\frac{1}{2d}}P_N(X)^\frac{1}{2} \rp\,.
\end{align*}
This estimate, together with Lemma~\ref{lemme comparaison diff symetrique discrete}, eventually yields
\begin{align*}
    \#(X\Delta(z'+B_{r_N}\cap \Zd)) \leq  C_{d}N\lp\alpha_N^\frac{1}{2}+N^{-\frac{1}{2d}}P_N(X)^\frac{1}{2} \rp\,,
\end{align*}
which concludes the proof for the $p$-capacity. \\
\noindent {\bf Step 2:} {\it relative capacity.}  Let $N\in\N$ and let $X\subset\Zd$ be such that $\#X=N$ and  \eqref{ineq:alphamin-R} holds, i.e.,
\begin{align}\label{ineq:minimumproof}
\mathrm{cap}_N(X,R) \leq m_N(R) +\alpha_N\,. 
\end{align}
By Lemma~\ref{lemme comparaison continu discret}, Lemma~\ref{lemme borne capacite min} and \eqref{ineq:minimumproof} we obtain
\begin{align}\label{ineq:quasiminproof}
\mathrm{cap}(\zeta(X),RN^{\frac{1}{d}})\leq \left(\frac{N}{|B_1|}\right)^{\frac{d-2}{d}} \left(\mathrm{cap}(B_1,R|B_1|^{\frac{1}{d}}) +C_d\alpha_N +C_{d} N^{-\frac{1}{d}}\right)\,.
\end{align}
We set $r=r_{|\zeta(X)|}=\lp\frac{|\zeta(X)|}{|B_1|}\rp^\frac{1}{d}$, so that $|r^{-1}\zeta(X)|=|B_1|$. By  \eqref{eq:assumptions-for-proof} we know that  $r^{-1} RN^{\frac{1}{d}} \in [R|B_1|^{\frac{1}{d}},2R|B_1|^{\frac{1}{d}}]$. Thanks to Lemma~\ref{lemme estimation mesure zeta}, \eqref{ineq:quasiminproof} and the scaling properties of the relative capacity, we infer that
\begin{align*}
\mathrm{cap}\left( r^{-1} \zeta(X),  r^{-1} RN^{\frac{1}{d}}\right)& = r^{2-d} \mathrm{cap}(\zeta(X),RN^{\frac{1}{d}}) \\&\leq r^{2-d}\left(\frac{N}{|B_1|}\right)^{\frac{d-2}{d}} \left(\mathrm{cap}(B_1,R|B_1|^{\frac{1}{d}}) +C_d\alpha_N +C_{d} N^{-\frac{1}{d}}\right)
\\& \leq \mathrm{cap}(B_1,r^{-1} RN^{\frac{1}{d}}) +C_d\alpha_N +C_{d} N^{-\frac{1}{d}}\mathrm{P}_N(X)\,.
\end{align*}
Due to \eqref{eq:assumptions-for-proof} and the assumption $\sup \alpha_N<+\infty$ we can assume that
\begin{align*}
 C_d\alpha_N +C_{d} N^{-\frac{1}{d}}\mathrm{P}_N(X) \leq 1\,.
\end{align*}
 We are in position to apply Lemma~\ref{lem:no-concentration-bdry} with $\Omega = r^{-1}\zeta(X)$, $R_0 = R|B_1|^{\frac{1}{d}}$, $\bar R = r^{-1} RN^{\frac{1}{d}} \in  [R|B_1|^{\frac{1}{d}},2R|B_1|^{\frac{1}{d}}]$, $\eta = 1$ and $\sigma =\frac{1}{3}|B_1|$ to obtain the existence of $\delta= \delta(d,R)$ such that  $|r^{-1} \zeta(X) \setminus B_{(1-\delta)\bar{R}}| \leq \frac{1}{3}|B_1|$ which implies that $|N^{-\frac{1}{d}}|B_1|^{\frac{1}{d}} \zeta(X) \setminus |B_1|^{\frac{1}{d}}B_{(1-\delta)R}| \leq \frac{1}{3}|B_1|$. We now apply Lemma~\ref{lem:diffiomorphism} with $\Omega= N^{-\frac{1}{d}}|B_1|^{\frac{1}{d}}\zeta(X)$, $\delta=\delta(d,R)>0$ as above, $\sigma=\frac{1}{3}|B_1|$, and $\varepsilon_N=|B_1|(1-N^{-1}|\zeta(X))|\leq  C_d N^{-\frac{1}{d}}P_N(X)$ and find a bi-Lipschitz map $\Phi_{N} \colon B_R \to B_R$ such that  
 \begin{align}\label{ineq:PhiN-lipschitz}
\|\nabla \Phi_{N} -\mathrm{Id}\|_{L^\infty(B_R)} \leq C_{d,R} N^{-\frac{1}{d}}P_N(X)  \,, \quad \quad  |\Phi_{N}(N^{-\frac{1}{d}}|B_1|^{\frac{1}{d}}\zeta(X))| = |B_1|\,, 
\end{align}
and by \eqref{lem-ineq:diff-cap}
\begin{align*}
\mathrm{cap}(\Phi_{N}(N^{-\frac{1}{d}}|B_1|^{\frac{1}{d}}\zeta(X)),R|B_1|^{\frac{1}{d}}) &\leq \mathrm{cap}(N^{-\frac{1}{d}}|B_1|^{\frac{1}{d}}\zeta(X),R|B_1|^{\frac{1}{d}}) +C_{d,R}N^{-\frac{1}{d}}P_N(X)\,.
\end{align*}
Combining this inequality with  \eqref{ineq:quasiminproof} and using the scaling property of the capacity we eventually obtain 
\begin{align*}
\mathrm{cap}(\Phi_{N}(N^{-\frac{1}{d}}\zeta(X))|B_1|^{\frac{1}{d}},R|B_1|^{\frac{1}{d}}) \leq \mathrm{cap}(B_1,R|B_1|^{\frac{1}{d}}) +C_d\alpha_N +C_{d,R} N^{-\frac{1}{d}}P_N(X)\,.
\end{align*}
Therefore, by the quantitative isoperimetric inequality in \cite{dephilippis2019sharpquantitativeisocapacitaryinequality}, we can write that 
\begin{align*}
 |\Phi_{N}(N^{-\frac{1}{d}}|B_1|^{\frac{1}{d}}\zeta(X))\Delta B_1| \leq C_{d,R}\alpha_N^{\frac{1}{2}} +C_{d,R} N^{-\frac{1}{2d}}P_N(X)^{\frac{1}{2}}\,.
\end{align*}
This together with \eqref{ineq:PhiN-lipschitz} implies
\begin{align*}
|\zeta(X)\Delta B_{r_N}| = C_{d}N |N^{-\frac{1}{d}}|B_1|^{\frac{1}{d}}\zeta(X)\Delta B_{1}| &\leq C_dN\left( |N^{-\frac{1}{d}}|B_1|^{\frac{1}{d}}\zeta(X)\Delta \Phi_N^{-1}\left( B_1\right)|  +C_{d,R} N^{-\frac{1}{d}}P_N(X)\right) \\&\leq C_dN\left( |\Phi_{N}(N^{-\frac{1}{d}}|B_1|^{\frac{1}{d}}\zeta(X))\Delta B_1| +C_{d,R} N^{-\frac{1}{d}}P_N(X)\right) \\&\leq C_{d,R}N\left(\alpha_N^{\frac{1}{2}} + N^{-\frac{1}{2d}}P_N(X)^{\frac{1}{2}}\right)\,.
\end{align*}
At this point, the proof can be completed by repeating the arguments used in the capacity case.
\end{proof}

\section{Discrete rearrangements and structure of optimal sets} \label{sec:disc-rear}

The main goal of this section is to obtain a uniform perimeter estimate for the optimal sets of the capacity problem (see Propostion~\ref{prop:perimeterestimate} for the precise statement). The proof uses some properties of discrete rearrangements that we introduce below.

\subsection{Discrete rearrengements} \label{subsec:discreteRearrangements}

In this section we introduce the definition of discrete (Schwarz-like) rearrangement. We first introduce the notion for functions defined on $\mathbb{Z}$ and then extend it to functions defined on $\mathbb{Z}^d$ (see \cite{schwarz,schwarzZd,Z2}).

\begin{definition}[Rearrangement of functions defined on $\mathbb{Z}$] \label{def:symmetrization1}
 Let $u\in \ell^2(\mathbb{Z})$. Its symmetrization $u^*\in \ell^2(\mathbb{Z})$ is the unique function obtained from $u$ by permuting its values and that satisfies for every $k \in \mathbb{N}$ the following property:
$$u^*(k)\geq u^*(-k)\geq u^*(k+1)\,.$$
\end{definition}
In order to extend the definition to $\mathbb{Z}^d$, we introduce the set of rearrangement directions 
\begin{align}\label{def:symdirections}
\mathcal{B}=\left\{ e_i \colon i \in \{1,\ldots,d\}  \right\} \cup \left\{ e_i \pm e_j \colon i,j \in \{1,\ldots,d\}\,, i \neq j \right\}\,.
\end{align}
Given $\xi \in \mathcal{B}$ we define 
\begin{align*}
\Pi_\xi= \begin{cases}
 \{\alpha \in \mathbb{Z}^d \colon \xi \cdot \alpha=0\}  &\text{if } \xi =e_j \text{ for some } j \in \{1,\ldots,d\}\,,\\ \{\alpha \in \mathbb{Z}^d \colon \xi \cdot \alpha=0\}  \cup  \{\alpha \in \mathbb{Z}^d \colon \xi \cdot \alpha=1\}  &\text{otherwise.}
 \end{cases}
\end{align*}
If $\xi \notin \{e_1,\ldots,e_d\}$ we additionally set $\Pi_\xi^0 = \{\alpha \in \mathbb{Z}^d \colon \xi \cdot \alpha=0\}$.
 Given $\xi \in \mathcal{B}$ and $\alpha \in \Pi_\xi$ we first define the function $u^{\alpha,\xi}\colon\Z\to\R$ as 
\begin{align}\label{def:urestrictedtoslice}
u^{\alpha,\xi}(t) = u(\alpha + t\xi)\,.  
\end{align}

\begin{definition}[Rearrangement of functions defined on $\mathbb{Z}^d$]\label{def:symmetrization-d}  Let $u\in \ell^2(\mathbb{Z^d})$. The symmetrization  of $u$ with respect to $\xi$ is denoted by $u^{*\xi}\colon\Zd\to\mathbb R$  and it is defined as 
\begin{align*}
u^{*\xi}(i) =\begin{cases}   (u^{\alpha,\xi})^*(t) &\text{if } i=\alpha+t\xi\,, \xi \cdot \alpha=0 \text{ for some } \alpha \in \Pi_\xi \text{ and } t \in \mathbb{Z}\,,\\
\mathrm{R}(u^{\alpha,\xi})^*(t) &\text{if } i=\alpha+t\xi\,, \xi \cdot \alpha=1 \text{ for some } \alpha \in \Pi_\xi \text{ and } t \in \mathbb{Z}\,,
\end{cases}
\end{align*}
where we denote by $\mathrm{R}u(i) = u(-i)$ the reflection of $u$ with respect to the origin. 

\end{definition}
For $u \in \ell^2(\mathbb{Z}^d)$ and $\xi_1,\xi_2 \in \mathcal{B}$ we define the iterative rearrangement of $u$ with respect to $\xi_1$ and $\xi_2$  as
\begin{align*}
u^{*(\xi_1,\xi_2)}=\big(u^{*\xi_1}\big)^{*\xi_2}\,.
\end{align*}
Given $n \in \mathbb{N}$ we set $ u^{*(\xi_1,\ldots,\xi_n)} = \left(u^{*(\xi_1,\ldots,\xi_{n-1})}\right)^{*\xi_n}$.\\

In the next Section~\ref{subsec:Proofperimeterestimate} we are going to prove that the rearrangement decreases the energy and to give necessary conditions for the equality case. To this end, we start in \ref{subsec:energydec} by rewriting the energy as sum of energies on one-dimensional slices, distinguishing between slices in coordinate directions and in diagonal directions.  

\subsection{Energy decomposition}\label{subsec:energydec} Let $u,v\in \ell^2(\mathbb{Z})$. We introduce the following notation:
 \begin{align} \label{def:interaction-energy}
  E^{1\mathrm{d}} (u,v)=\sum_{i\in\Z}|u(i)-v(i)|^2
 \end{align}
and
\begin{align}\label{def:onedimensionalenergy}
E^{1\mathrm{d}}(u)= \sum_{i \in \mathbb{Z}} |u(i+1)-u(i)|^2\,.
\end{align}
Furthermore we set
\begin{align}\label{def:onedimensionaldiagonal}
 E_{\mathrm{diag}}^{1\mathrm{d}} (u,v)=\sum_{i\in\Z}|u(i)-v(i)|^2 + \sum_{i\in\Z}|u(i+1)-v(i)|^2\,.
 \end{align}
With the above notation the following useful lemma holds true.

\begin{lemma}(Energy decomposition)\label{lem:decomposition} Let $u \in \ell^2(\mathbb{Z}^d)$ and $\xi \in \mathcal{B}$. Then the following statements are true:
\begin{itemize}
\item[(1)] If $\xi=e_j$ for some $j\in \{1,\ldots,d\}$, then
\begin{align*}
E(u)= \sum_{\alpha \in \Pi_{\xi}} E^{1\mathrm{d}}(u^{\alpha,\xi}) +  \underset{k\neq j}{\sum_{k=1}^d} \sum_{\alpha \in \Pi_{\xi}} E^{1\mathrm{d}}(u^{\alpha,\xi},u^{\alpha+e_k,\xi})\,.
\end{align*}
\item[(2)] If $\xi=e_j\pm e_l$ for $j,l \in \{1,\ldots,d\}$ such that $j\neq l$, then
\begin{align*}
E(u)&= \sum_{\alpha \in \Pi_\xi^0} \left(E^{1\mathrm{d}}_{\mathrm{diag}}(u^{\alpha,\xi},u^{\alpha+e_j,\xi})+E^{1\mathrm{d}}_{\mathrm{diag}}(u^{\alpha,\xi},u^{\alpha+e_l,\xi}) \right)\\ &\quad+ \underset{k\notin \{j,l\}}{\sum_{k=1}^d} \sum_{\alpha \in \Pi_{\xi}^0}\left( E^{1\mathrm{d}}(u^{\alpha,\xi},u^{\alpha+e_k,\xi})+E^{1\mathrm{d}}(u^{\alpha+e_j,\xi},u^{\alpha+e_j+e_k,\xi}) \right) \,.
\end{align*}
\end{itemize}

\end{lemma}
\begin{proof} We prove the two cases separately.

\noindent {\bf Step 1:} {\it Proof of {\rm (1)}:}
Assume that $\xi=e_j$ for some $j \in \{1,\ldots,m\}$. Observe that, by Fubini's Theorem, \eqref{def:urestrictedtoslice}--\eqref{def:onedimensionalenergy}, we have
\begin{align*}
E(u) &= \sum_{k=1}^d \sum_{i \in \mathbb{Z}^d} |u(i+e_k)-u(i)|^2 = \sum_{k=1}^d  \sum_{\alpha \in \Pi_{e_j}} \sum_{t \in \mathbb{Z}} |u(\alpha+te_j +e_k)-u(\alpha+te_j)|^2 \\&=  \sum_{\alpha \in \Pi_{e_j}} \sum_{t \in \mathbb{Z}} |u^{\alpha,e_j}(t+1)-u^{\alpha,e_j}(t)|^2 + \underset{k\neq j}{\sum_{k=1}^d} \sum_{\alpha \in \Pi_{e_j}} \sum_{t \in \mathbb{Z}} |u^{\alpha+e_k,e_j}(t)-u^{\alpha,e_j}(t)|^2 \\&=  \sum_{\alpha \in \Pi_{e_j}} E^{1\mathrm{d}}(u^{\alpha,e_j}) +  \underset{k\neq j}{\sum_{k=1}^d} \sum_{\alpha \in \Pi_{e_j}} E^{1\mathrm{d}}(u^{\alpha,e_j},u^{\alpha+e_k,e_j}) \,.
\end{align*}
This concludes the proof of (1). \\
\noindent {\bf Step 2:} {\it Proof of {\rm (2)}:} Assume that $\xi=e_j+e_l$ for some $j,k\in \{1,\ldots,m\}$ such that $j\neq l$. Again observe that, by Fubini's Theorem, \eqref{def:urestrictedtoslice},\eqref{def:onedimensionalenergy} and , we have
\begin{align*}
E(u) &= \sum_{k=1}^d \sum_{i \in \mathbb{Z}^d} |u(i+e_k)-u(i)|^2 = \sum_{k=1}^d  \sum_{\alpha \in \Pi_{\xi}} \sum_{t \in \mathbb{Z}} |u(\alpha+t\xi +e_k)-u(\alpha+t\xi)|^2 \\&= \sum_{k\in \{j,l\}} \sum_{\alpha \in \Pi_{\xi}} \sum_{t \in \mathbb{Z}} |u(\alpha+t\xi +e_k)-u(\alpha+t\xi)|^2+   \underset{k\notin \{j,l\}}{\sum_{k=1}^d} \sum_{\alpha \in \Pi_{\xi}} \sum_{t \in \mathbb{Z}} |u^{\alpha+e_k,\xi}(t)-u^{\alpha,\xi}(t)|^2 \\&= \sum_{\alpha \in \Pi_\xi^0} \left(E^{1\mathrm{d}}_{\mathrm{diag}}(u^{\alpha,\xi},u^{\alpha+e_j,\xi})+E^{1\mathrm{d}}_{\mathrm{diag}}(u^{\alpha,\xi},u^{\alpha+e_l,\xi}) \right)+ \underset{k\notin \{j,l\}}{\sum_{k=1}^d} \sum_{\alpha \in \Pi_{\xi}} E^{1\mathrm{d}}(u^{\alpha,\xi},u^{\alpha+e_k,\xi}) \\&= \sum_{\alpha \in \Pi_\xi^0} \left(E^{1\mathrm{d}}_{\mathrm{diag}}(u^{\alpha,\xi},u^{\alpha+e_j,\xi})+E^{1\mathrm{d}}_{\mathrm{diag}}(u^{\alpha,\xi},u^{\alpha+e_l,\xi}) \right) \\&\quad+\underset{k\notin \{j,l\}}{\sum_{k=1}^d} \sum_{\alpha \in \Pi_{\xi}^0} \left( E^{1\mathrm{d}}(u^{\alpha,\xi},u^{\alpha+e_k,\xi})+E^{1\mathrm{d}}(u^{\alpha+e_j,\xi},u^{\alpha+e_j+e_k,\xi}) \right) \,.
 \end{align*}
 The proof of the case of $\xi=e_j -e_l$ follows along the same lines. This concludes the proof of~(2).
\end{proof}

\subsection{The perimeter bound} \label{subsec:Proofperimeterestimate}

In what follows we provide the statement of the uniform perimeter bound for optimal sets of the isocapacitary problem (Proposition~\eqref{prop:perimeterestimate}). The section contains the main energy estimates and geometric properties of the equilibrium (capacitary) potential necessary for the proof, which is postponed to the next Section~\eqref{sec:diameter estimate}, which contains the proof of the diameter estimate of optimal sets.

\begin{proposition}\label{prop:perimeterestimate} Let $X \subset \mathbb{Z}^d$ with $\#X=N$.
 \begin{itemize}
 \item[(1)] \label{ineqprop:p-capacityperimeter} For the capacity:
 \begin{equation*}
\text{if } \quad  \cappn{X}= m_{N}\,, \quad\text{then}\quad  P_N(X) \leq C_{d}
 \end{equation*}
  for some constant $C_{d}>0$.
     \item[(2)]\label{ineq:alphamin-R} For the relative capacity: 
 \begin{equation*} 
\text{if } \quad \caprn{X}= m_{N}(R)\,, \quad\text{then}\quad   P_N(X) \leq C_{d,R}
 \end{equation*}
 for some constant $C_{d,R}>0$.
\end{itemize}
\end{proposition}

In order to prove the previous proposition we make use of the following lemma.

\begin{lemma}\label{lem:single slice}
    Let $u\in \ell^2(\mathbb{Z})$ such that $u\geq 0$. Then
    \begin{align}\label{ineqlem:symmetrization2}
    E^{1\mathrm{d}}(\mathrm{R}u^*)=  E^{1\mathrm{d}}(u^*)\leq E^{1\mathrm{d}}(u)\,.
    \end{align}
    Equality holds only if $u$ is a decreasing function about its maximum. In particular, for all $t \in \mathbb{R}$ the set $\{u \geq t\}$ and $\{u>t\}$ are of the form $I\cap \mathbb{Z}$, where $I$ is an interval. 
\end{lemma}

\begin{proof} We divide the proof into two steps. First, we show that if  $u$ is not a decreasing function about its maximum we can (strictly) decrease the energy of $u$ by permuting its values. Second, we show that, for a function $u$ decreasing about its maximum, \eqref{ineqlem:symmetrization2} holds.

We first observe that, since the function $t\mapsto (t-b)^2+(t-c)^2$ is strictly increasing for $ t> \max\{b,c\}$, there holds
\begin{align}\label{ineq:abc}
(a-b)^2+(a-c)^2 > |b-c|^2\quad \text{for all }  a,b,c \in \mathbb{R} \quad \text{with } a> \max\{b,c\}\,.
\end{align} 
 Let $u \in \ell^2(\mathbb{Z})$ and let $u(\mathbb Z)=\{\alpha_k\}_{k\in \mathbb{N}}$ with $\alpha_k > \alpha_{k+1}$ for all $k \in \mathbb{N}$. \\
\noindent {\bf Step 1:} {\it Reduction to functions $u$ decreasing about their maximum:} We construct $u_{\mathrm{dec}} \in \ell^2(\mathbb{Z})$ such that $u_{\mathrm{dec}}(\mathbb Z)=u(\mathbb Z)$, $u_{\mathrm{dec}}$ is a decreasing function about its maximum and 
\begin{align}\label{ineq:udecu}
E^{1\mathrm{d}}(u_{\mathrm{dec}}) \leq E^{1\mathrm{d}}(u)
\end{align}
with equality only if $u$ is itself a decreasing function about its maximum.  We note that claiming that a function $u$ is decreasing about its maximum is equivalent to claiming that for all $k \in \mathbb{Z}$ there holds
\begin{align}\label{impl:decreasing}
\begin{split}
&u(k) = \max\{u(i)\colon i \geq k\} \quad \implies \quad u(k+1) =  \max\{u(i)\colon i \geq k+1\} \\
&u(k) = \max\{u(i)\colon i \leq k\} \quad \implies \quad u(k-1) =  \max\{u(i)\colon i \leq k-1\}\,.
\end{split}
\end{align} 
 We now construct $u_{\mathrm{dec}}$ and show \eqref{ineq:udecu}. We can assume that \eqref{impl:decreasing} is not true, otherwise $u_{\mathrm{dec}}=u$ suffices. Up to translation and reflection, it suffices to prove only the first implication for $k=0$. Hence we are left to construct $u_0 \in \ell^2(\mathbb{Z})$ such that $u_{0}(\mathbb Z)=u(\mathbb Z)$, it  satisfies
 \begin{align}\label{ineq:u0u}
 E^{1\mathrm{d}}(u_{0}) < E^{1\mathrm{d}}(u)\,,
 \end{align}
  and
 \begin{align*}
&u(0) = \max\{u(i)\colon i \geq 0\} \quad \implies \quad u(1) =  \max\{u(i)\colon i \geq 1\} \,.
\end{align*}
Now assume $u(0)=\alpha_{M}= \max\{u(i)\colon i\geq 0\}$, $u(j_0)=\alpha_{m}=\max\{u(i) \colon i \geq 1\}$, where $j_0=\min\{k\geq 2\colon u(k)=\alpha_m\}$, and $u(1)=\alpha_l<\alpha_m$. Finally, suppose that $u(k)= \alpha_m$ for all $k \in \{j_0,\ldots,j_0+l\}$, $u(j_0+l+1)=\alpha_n<\alpha_m$, and $u(j_0-1)=\alpha_r <\alpha_m$. We then define $u_0\colon\Z\to \R$ (see Figure~\ref{Fig:shiftv} for an illustration) as 
\begin{align*}
  u_0(k)=
    \begin{cases}
        u(k)  &\text{if }k\leq 0\textrm{ or }k\geq j_0+l+1\,,\\
        \alpha_m  &\text{if } 1\leq k\leq l+1 \,,\\
        u(k-l-1)  &\textrm{else.}
    \end{cases}
\end{align*}
We observe that
\begin{align}\label{eq:energydifference}
\begin{split}
E^{1\mathrm{d}}(u)-E^{1\mathrm{d}}(u_0)=(\alpha_M-\alpha_l)^2+(\alpha_m-\alpha_n)^2&+(\alpha_m-\alpha_r)^2 \\& - (\alpha_M-\alpha_m)^2 -(\alpha_m-\alpha_l)^2 -|\alpha_l-\alpha_r|^2\,.
\end{split}
\end{align}
Now, by the superadditivity of the function $t \mapsto |t|^2$ on $t \geq 0$ we have 
\begin{align*}
 (\alpha_M-\alpha_l)^2 -(\alpha_M-\alpha_m)^2-(\alpha_m-\alpha_l)^2 \geq 0\,.
\end{align*}
Additionally, thanks to \eqref{ineq:abc}, we have 
\begin{align*}
(\alpha_m-\alpha_n)^2+(\alpha_m-\alpha_r)^2 -|\alpha_l-\alpha_r|^2>0\,.
\end{align*}
 The last two inequalitites together with \eqref{eq:energydifference} show \eqref{ineq:u0u}. Now fixing $i_0$ such that $u(i_0)=\alpha_1=\max\{u(i)\colon i \in \mathbb{Z}\}$ we can construct $u_{\mathrm{dec}}$ by iterating the above procedure for all $k \in \mathbb{Z}$ (starting from $i_0$ increasingly resp.~decreasingly).

\begin{figure}[htp]
\begin{tikzpicture}[scale=.6]

\tikzset{>={Latex[width=1mm,length=1mm]}};

\draw(0,3) node[anchor=east]{$u$};
\draw[->,gray](3.25,1) to[out=-10, in=190] (15.75,1);

\draw[ultra thin, gray!50!black,->](-1,0)--++(9,0);
\draw[ultra thin, gray!50!black,->](0,-1)--++(0,4);

\draw[ultra thin, gray, dashed] (0,1.8)--++(8,0);
\draw (0,1.8) node[anchor=east] {$\alpha_m$};

\draw (0,2.2) node[anchor=east] {$\alpha_M$};

\draw[fill=black] (0,0) circle(.05);
\draw[fill=black] (0,2.2) circle(.05);
\draw(0,0) node[anchor=north east] {$0$};

\draw(0,2.2)--(1,1.3);

\draw[fill=black] (1,0) circle(.05);
\draw(1,0) node[anchor=north] {$1$};
\draw[fill=black] (1,1.3) circle(.05);

\draw[fill=black] (2,0) circle(.05);
\draw[fill=black] (2,1.4) circle(.05);

\draw[fill=black] (3,0) circle(.05);
\draw[fill=black] (3,1.2) circle(.05);
\draw(3,1.2)--(4,1.8);

\draw[fill=black] (4,0) circle(.05);
\draw[fill=black] (4,1.8) circle(.05);
\draw(4,0) node[anchor=north] {$j_0$};

\draw[fill=black] (6,0) circle(.05);
\draw[fill=black] (6,1.8) circle(.05);
\draw[fill=black] (5,0) circle(.05);
\draw[fill=black] (5,1.8) circle(.05);
\draw(6,0) node[anchor=north] {$j_0+l$};

\draw[fill=black] (7,0) circle(.05);
\draw[fill=black] (7,1.5) circle(.05);
\draw(7,1.5)--(6,1.8);

\draw[gray](4,1.6)--++(2,0) arc(-90:90:.2)--++(-2,0)arc(90:270:.2);

\draw[gray](1,1.1)--++(2,0) arc(-90:90:.2)--++(-2,0)arc(90:270:.2);

\draw[->,gray](6.25,1.5) to[out=-15, in=195] (12.75,1.5);


\begin{scope}[shift={(12,0)}]

\draw[ultra thin, gray!50!black,->](-1,0)--++(9,0);
\draw[ultra thin, gray!50!black,->](0,-1)--++(0,4);

\draw(0,3) node[anchor=east]{$u_0$};

\draw[ultra thin, gray, dashed] (0,1.8)--++(8,0);
\draw (0,1.8) node[anchor=east] {$\alpha_m$};

\draw (0,2.2) node[anchor=east] {$\alpha_M$};

\draw[fill=black] (0,0) circle(.05);
\draw[fill=black] (0,2.2) circle(.05);
\draw(0,0) node[anchor=north east] {$0$};
\draw(0,2.2)--(1,1.8);

\draw[fill=black] (1,0) circle(.05);
\draw[fill=black] (1,1.8) circle(.05);

\draw[fill=black] (2,0) circle(.05);
\draw[fill=black] (2,1.8) circle(.05);

\draw[fill=black] (3,0) circle(.05);
\draw[fill=black] (3,1.8) circle(.05);
\draw(3,0) node[anchor=north] {$l+1$};

\draw(3,1.8)--(4,1.3);

\draw[fill=black] (4,0) circle(.05);
\draw[fill=black] (4,1.3) circle(.05);

\draw[fill=black] (6,0) circle(.05);
\draw[fill=black] (6,1.3) circle(.05);
\draw[fill=black] (5,0) circle(.05);
\draw[fill=black] (5,1.4) circle(.05);
\draw(6,0) node[anchor=north] {$j_0+l$};

\draw[fill=black] (7,0) circle(.05);
\draw[fill=black] (7,1.5) circle(.05);
\draw(7,1.5)--(6,1.3);

\draw[gray](1,1.6)--++(2,0) arc(-90:90:.2)--++(-2,0)arc(90:270:.2);

\draw[gray](4,1.1)--++(2,0) arc(-90:90:.2)--++(-2,0)arc(90:270:.2);

\end{scope}
\end{tikzpicture}
\caption{Construction of the function $u_0$ from $u$: The connected component of $\{u=\alpha_m\}$ is shifted to be adjacent to $0$ and the values in between are shifted to start where the new component $\{u_0=\alpha_m\}$ ends.}
\label{Fig:shiftv}
\end{figure}
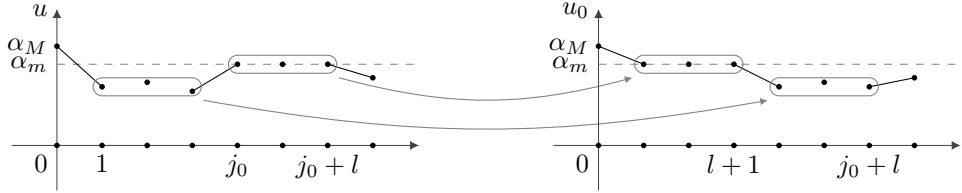

\noindent {\bf Step 2:} {\it Inequality with respect to symmetrization:} We want to show that
\begin{align}\label{ineq:ustaru-step2}
E^{1\mathrm{d}}(u^*)\leq E^{1\mathrm{d}}(u)\,.
\end{align}
By Step 1, it suffices to consider the case where $u$ is decreasing about its maximum. Furthermore, by a perturbation argument we can additionally assume that $u$ is injective. Indeed, if this is not the case, for every $\varepsilon>0$ we can find a function $v_\varepsilon \colon \mathbb{Z} \to \mathbb{R}$  such that
\begin{itemize}
\item[(i)] It holds
\begin{align*}
\sum_{i \in \mathbb{Z}} |v_\varepsilon(i)|^2 <\varepsilon\,;
\end{align*}
\item[(ii)] The function $v_\varepsilon$ is injective and decreasing around the maximum of $u$;
\item[(iii)] For all $i \in \mathbb{Z}$ there holds
\begin{align*}
u(i) = \alpha_k \quad \implies \quad v_\varepsilon(i) <\frac{1}{2} \min\{|\alpha_k-\alpha_{k+1}|,|\alpha_k-\alpha_{k-1}|\}\,.
\end{align*}
\end{itemize}
We now set $u_\varepsilon(i) = u(i) + v_\varepsilon(i)$ and observe that, as $v_\varepsilon$ satisfies (ii) and (iii), $u_\varepsilon$ is injective and decreasing about its maximum. Moreover, thanks to (i) and (ii), we have
\begin{align*}
\lim_{\varepsilon \to 0} E^{1\mathrm{d}}(u_\varepsilon)=E^{1\mathrm{d}}(u) \quad \text{ and } \quad \lim_{\varepsilon \to 0} E^{1\mathrm{d}}((u_\varepsilon)^*)=E^{1\mathrm{d}}(u^*)\,. 
\end{align*}
As a consequence, we need to prove \eqref{ineq:ustaru-step2} only in the case where $u$ is monotonically decreasing about its maximum and injective. As before, denote by $u(\mathbb{Z}) =\{\alpha_k\}_{k \in \mathbb{N}}$, with $\alpha_{k} \geq \alpha_{k+1}$ and assume without loss of generality that $u(0)=\alpha_1$. In this case 
\begin{align*}
u^*(i) = \begin{cases} \alpha_{2i} &\text{if } i> 0\,,\\
\alpha_{-2i+1} &\text{if } i \leq0\,.
\end{cases}
\end{align*}
We show \eqref{ineq:ustaru-step2} with equality only if (after a possible reflection with respect the origin) $u(i)=u^*(i)$. We proceed by induction and assume that $\{u \geq \alpha_{k-1}\} =\{u^* \geq \alpha_{k-1}\}$ (up to reflection, this is true for $k \in \{1,2\}$). Now assume, there is $k \in \mathbb{N}$ such that $\{u \geq \alpha_{k-1}\} =\{u^* \geq \alpha_{k-1}\}$ but $\{u \geq \alpha_{k}\} \neq \{u^* \geq \alpha_{k}\}$. For simplicity we assume that $k$ is even (thus $k\geq 4$, the construction for $k$ odd is similar) and define $u_k \colon \mathbb{Z} \to \mathbb{R}$ (see Figure~\ref{Fig:constructionuk} for an illustration) by
\begin{align*}
u_k(i) =\begin{cases} u(i) &\text{if }  u(i) \geq \alpha_{k-1}\,,\\
u(-i) &\text{if } u(i) \leq \alpha_{k}\,.
\end{cases}
\end{align*}
As $u$ is decreasing about its maximum we have that $\{u_k \geq \alpha_{k}\} = \{u^* \geq \alpha_{k}\}$ and, denoting by $u(\frac{k}{2}) =\alpha_m <\alpha_k$, we have
\begin{align*}
E^{1\mathrm{d}}(u)-E^{1\mathrm{d}}(u_k) = |\alpha_{k-2}-\alpha_{m}|^2 +  |\alpha_{k-1}-\alpha_{k}|^2 - |\alpha_{k-1}-\alpha_m|^2- |\alpha_{k-2}-\alpha_k|^2\,.
\end{align*}
As $\alpha_m<\alpha_k <\alpha_{k-1} <\alpha_{k-2} $ it is elementary to verify that   
\begin{align*}
|\alpha_{k-2}-\alpha_{m}|^2 +  |\alpha_{k-1}-\alpha_{k}|^2 - |\alpha_{k-1}-\alpha_m|^2- |\alpha_{k-2}-\alpha_k|^2 >0\,.
\end{align*}
Thus, noting that $E^{1\mathrm{d}}(u_k)$ converges decreasingly to $E^{1\mathrm{d}}(u^*)$ as $k \to +\infty$, this concludes the proof.
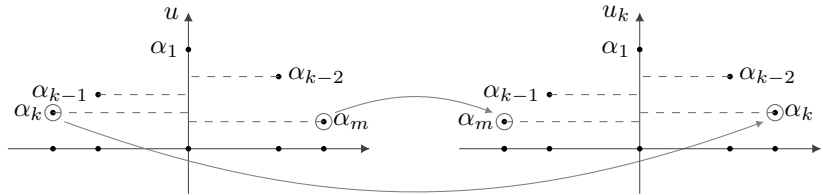
\begin{figure}[htp]
\begin{tikzpicture}[scale=.6]
\tikzset{>={Latex[width=1mm,length=1mm]}};

\draw[ultra thin, gray!50!black,->](0,-1)--++(0,4);
\draw[ultra thin, gray!50!black,->](-4,0)--++(8,0);

\draw(0,3) node[anchor=east]{$u$};

\draw[fill=black](0,0) circle(.05);

\draw[fill=black](0,2.2)circle(.05);

\draw(0,2.2) node[anchor=east]{$\alpha_1$};


\draw[fill=black](2,0) circle(.05);

\draw[fill=black](2,1.6) circle(.05);

\draw(2,1.6) node[anchor=west]{$\alpha_{k-2}$};

\draw[dashed,ultra thin,gray!80!black](2,1.6)--++(-2,0); 

\draw[fill=black](3,0) circle(.05);

\draw[fill=black](3,.6) circle(.05);

\draw(3,.6) node[anchor=west]{$\alpha_{m}$};

\draw[dashed,ultra thin,gray!80!black](3,.6)--++(-3,0); 


\draw[fill=black](-2,0) circle(.05);

\draw[fill=black](-2,1.2) circle(.05);

\draw(-2,1.2) node[anchor=east]{$\alpha_{k-1}$};

\draw[dashed,ultra thin,gray!80!black](-2,1.2)--++(2,0); 

\draw[fill=black](-3,0) circle(.05);

\draw[fill=black](-3,.8) circle(.05);

\draw(-3,.8) node[anchor=east]{$\alpha_{k}$};

\draw[dashed,ultra thin,gray!80!black](-3,.8)--++(3,0); 

\draw[gray!80!black](-3,.8) circle(.18);

\draw[->,gray](-2.75,.6) to[out=-20, in=200] (12.75,.6);

\draw[->,gray](3.25,.8) to[out=20, in=160] (6.75,.8);

\draw[gray!80!black](3,.6) circle(.18);


\begin{scope}[shift={(10,0)}]

\draw[ultra thin, gray!50!black,->](0,-1)--++(0,4);
\draw[ultra thin, gray!50!black,->](-4,0)--++(8,0);

\draw(0,3) node[anchor=east]{$u_k$};

\draw[fill=black](0,0) circle(.05);

\draw[fill=black](0,2.2)circle(.05);

\draw(0,2.2) node[anchor=east]{$\alpha_1$};


\draw[fill=black](2,0) circle(.05);

\draw[fill=black](2,1.6) circle(.05);

\draw(2,1.6) node[anchor=west]{$\alpha_{k-2}$};

\draw[dashed,ultra thin,gray!80!black](2,1.6)--++(-2,0); 

\draw[fill=black](-3,0) circle(.05);

\draw[fill=black](-3,.6) circle(.05);

\draw(-3,.6) node[anchor=east]{$\alpha_{m}$};

\draw[dashed,ultra thin,gray!80!black](-3,.6)--++(3,0); 


\draw[fill=black](-2,0) circle(.05);

\draw[fill=black](-2,1.2) circle(.05);

\draw(-2,1.2) node[anchor=east]{$\alpha_{k-1}$};

\draw[dashed,ultra thin,gray!80!black](-2,1.2)--++(2,0); 

\draw[fill=black](-3,0) circle(.05);

\draw[fill=black](3,0) circle(.05);

\draw[fill=black](3,.8) circle(.05);

\draw(3,.8) node[anchor=west]{$\alpha_{k}$};

\draw[dashed,ultra thin,gray!80!black](3,.8)--++(-3,0); 

\draw[gray!80!black](3,.8) circle(.18);

\draw[gray!80!black](-3,.6) circle(.18);

\end{scope}

\end{tikzpicture}
\caption{The function $u_k$ obtained from $u$: $\{u_k \leq \alpha_k\} = -\{u \leq \alpha_k\}$. }
\label{Fig:constructionuk}
\end{figure} 
\end{proof}

The next elementary lemma, whose proof is left to the reader will be used several times in the rest of this section.

\begin{lemma} \label{lem:minmax} Let $a_1, a_2,b_1,b_2 \geq 0$. Then 
\begin{align}\label{ineq:maxminpower-p}
|a_1\wedge a_2-b_1\wedge b_2|^2 + |a_1\vee a_2-b_1\vee b_2|^2 \leq |a_1-b_1|^2 +  |a_2-b_2|^2
\end{align}
with equality only if
\begin{align}\label{ineq:equalitycasepowerp}
(a_1-a_2)(b_1-b_2)\geq 0\,.
\end{align}
\end{lemma}

\begin{lemma} \label{lem:neighbouring-lines}
    Let $u,v\in \ell^2(\mathbb{Z})$ be such that $u,v \geq 0$. Then
    \begin{equation}\label{ineqlem:symmetrization}
    E^{1\mathrm{d}}(\mathrm{R}u^*,\mathrm{R}v^*) =  E^{1\mathrm{d}}(u^*,v^*)\leq  E^{1\mathrm{d}}(u,v)\,.
    \end{equation}
Equality holds only if there exists a bijection
$
i \colon \mathbb{N} \to \mathbb{Z}
$
such that 
\begin{align}\label{eq:monotone-reordering}
u(i(k+1)) \le u(i(k)) 
\quad\text{ and }\quad
v(i(k+1)) \le v(i(k))
\qquad \text{for all } k \ge 1\,.
\end{align}
\end{lemma}

\begin{proof}  
Assume $u(\mathbb{Z}) = \{\alpha_k\}_{k\in \mathbb{N}}$ with $\alpha_k \geq \alpha_{k+1} \geq 0$ for all $k \in \mathbb{N}$ and $v(\mathbb{Z}) = \{\beta_k\}_{k\in \mathbb{N}}$ with $\beta_k \geq \beta_{k+1} \geq 0$ for all $k \in \mathbb{N}$. Fix $i_1,j_1 \in \mathbb{Z}$ such that $u(i_1)=\alpha_1$ and $v(j_1)=\beta_1$. Then define $v_1\colon \mathbb{Z} \to \mathbb{R}$ by
\begin{align*}
v_1(i) = \begin{cases} \beta_1 &\text{if } i=i_1\,,\\
v(i_1) & \text{if } i=j_1\,,\\
v(i) &\text{otherwise.}
\end{cases}
\end{align*}
Thus, by Lemma~\ref{lem:minmax}
\begin{align*}
E^{1\mathrm{d}}(u,v_1)\leq E^{1\mathrm{d}}(u,v)
\end{align*}
with equality only if either $u(j_1) = \alpha_1$ or $v(i_1) = \beta_1$. Assume that $v_{n-1}$ has been constructed, where $u(i_k)=\alpha_k$ and $v_{n-1}(i_k)=\beta_k$ for all $1\leq k \leq n-1$. By repeating the above construction with $i_n,j_n \in \mathbb{Z} \setminus \{i_1,\ldots,i_{n-1}\}$ such that $u(i_n)=\alpha_n$ and $v(j_n)=\beta_n$ we inductively construct $v_n$ and $i\colon  \mathbb{N}\to \mathbb{Z}$, $k\mapsto i(k):=i_k$ such that
\begin{align*}
u(i(k+1)) \le u(i(k)) 
\quad\text{ and }\quad
v_n(i(k+1)) \le v_n(i(k))
\qquad \text{for all } 1\leq k\leq n\,.
\end{align*}
and 
\begin{align*}
E^{1\mathrm{d}}(u,v_n)\leq E^{1\mathrm{d}}(u,v_{n-1}). 
\end{align*}
Eventually the induction gives that 
\begin{align}\label{lem:neighbouring-lines_ineq}
E^{1\mathrm{d}}(u,v_n)\leq E^{1\mathrm{d}}(u,v), 
\end{align}
with equality only if 
\begin{align*}
u(i(k+1)) \le u(i(k)) 
\quad\text{ and }\quad
v(i(k+1)) \le v(i(k))
\qquad \text{for all } 1\leq k\leq n\,.
\end{align*}
As $u^*$ and $v^*$ are aligned monotonically, by Definition~\ref{def:symmetrization1}, we have that $E^{1\mathrm{d}}(u,v_n)$ converges decreasingly to $E^{1\mathrm{d}}(u^*,v^*)$ and the claim of the lemma follows from \eqref{lem:neighbouring-lines_ineq} and the characterization of the equality case.
\end{proof}

\begin{figure}[htp]
            \begin{tikzpicture}[description/.style={fill=white,inner sep=2pt}]
            
 \draw[fill=black](-1.975,.5) circle(.05);
\draw[fill=black](-1.975,-.5) circle(.05);
\draw[fill=black](-3.825,-.5) circle(.05);
\draw[fill=black](-3.825,.5) circle(.05);

\draw[fill=black](-.3,-.5) circle(.05);
\draw[fill=black](-.3,.5) circle(.05);
\draw[fill=black](1.675,.5) circle(.05);
\draw[fill=black](1.675,-.5) circle(.05);
\draw[fill=black](3.8125,-.5) circle(.05);
\draw[fill=black](3.8125,.5) circle(.05);
                    \matrix (m) [matrix of math nodes, row sep=3em,
                    column sep=2.5em, text height=1.5ex, text depth=0.5ex]
                    { v(i_1)  & v(i_2) & \cdots & v(i_{n-1}) & v(i_n) \\
       u(1)     & u(2)      & \cdots & u(n-1)  & u(n)                \\};
                    \path[-,font=\scriptsize]
                    (m-1-1) edge node[auto] {} (m-2-1)
                    (m-1-2) edge node[auto] {} (m-2-2)
                    (m-1-3) edge node[auto] {} (m-2-3)
                    (m-1-4) edge node[auto] {} (m-2-4)
                    (m-1-5) edge node[auto] {} (m-2-5);
            \end{tikzpicture}
            \caption{The rearrangement of $v$ with respect to $u$: The highest value of $v$ is rearranged in order to be paired with the highest value of $u$.}
            \label{Fig:Euv} 
\end{figure}
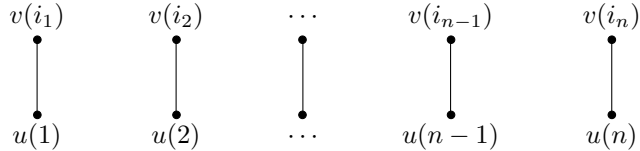

\begin{lemma}\label{lem:diagonalsym} Let $u, v \in \ell^2(\mathbb{Z})$ be such that $u,v \geq 0$. Then
\begin{align}\label{ineqlem:diagonalsymmetrization}
E^{1\mathrm{d}}_{\mathrm{diag}}(u^*,\mathrm{R} v^*) \leq E^{1\mathrm{d}}_{\mathrm{diag}}(u,v)\,.
\end{align}
\end{lemma}
\begin{proof} We divide the proof into three steps. In the first step we introduce two transformations of the functions $u$, $v$ and calculate their effect on the energy.   \\
\noindent {\bf Step 1:} {\it The flip operators:} Let $w \in \ell^2(\mathbb{Z})$ and $l,m \in \mathbb{Z}$ with $l<m$. We define $w_{l,m} \in \ell^2(\mathbb{Z})$ as
\begin{align*}
w_{l,m}(i) = \begin{cases} w(m+l-i) &\text{if } l\leq i \leq m\,,\\
w(i) &\text{otherwise.}\\
\end{cases}
\end{align*}
It is then easy to verify (see Figure~\ref{Fig:diagmin}) that for $u,v \in \ell^2(\mathbb{Z})$ and $l,m \in \mathbb{Z}$ with $l<m$ we have
\begin{align}\label{ineq:flip1}
\begin{split}
E^{1\mathrm{d}}_{\mathrm{diag}}(u_{l+1,m},v_{l,m}) = E^{1\mathrm{d}}_{\mathrm{diag}}(u,v) &+|u(l)-v(m)|^2 +|u(m+1)-v(l)|^2 \\&-  |u(l)-v(l)|^2 - |u(m+1)-v(m)|^2
\end{split}
\end{align}
and for $l,m \in \mathbb{Z}$ with $l< m$ we have
\begin{align}\label{ineq:flip2}
\begin{split}
E^{1\mathrm{d}}_{\mathrm{diag}}(u_{l,m},v_{l,m-1}) = E^{1\mathrm{d}}_{\mathrm{diag}}(u,v) &+|u(m)-v(l-1)|^2 +|u(l)-v(m)|^2 \\&-  |u(l)-v(l-1)|^2 - |u(m)-v(m)|^2 \,.
\end{split}
\end{align} 
In the following we set $u(\mathbb{Z}) = \{\alpha_k\}_{k \in \mathbb{N}}$ with $\alpha_{k} > \alpha_{k+1}\geq 0$ for all $k \in \mathbb{N}$ and $v(\mathbb{Z}) = \{\beta_k\}_{k \in \mathbb{N}}$ with $\beta_{k} > \beta_{k+1}\geq 0$ for all $k \in \mathbb{N}$.

\noindent {\bf Step 2:} {\it Reduction to centered functions $u,v$ decreasing about their maximum:} In this step we construct $u_{\mathrm{dec}}, v_{\mathrm{dec}} \in \ell^2(\mathbb{Z})$ such that $u_{\mathrm{dec}}, v_{\mathrm{dec}}$, take the same values as $u,v$ respectively, $u_{\mathrm{dec}}(0) = \max\{u(i)\colon i \in \mathbb{Z}\}$, $v_{\mathrm{dec}}(0) = \max\{v(i)\colon i \in \mathbb{Z}\}$, $u_{\mathrm{dec}}, v_{\mathrm{dec}}$ are decreasing about their maxima, and
\begin{align}\label{ineq:diagdecreasing}
E^{1\mathrm{d}}_{\mathrm{diag}}(u_{\mathrm{dec}}, v_{\mathrm{dec}}) \leq E^{1\mathrm{d}}_{\mathrm{diag}}(u,v)\,. 
\end{align}
Without loss of generality, assume that $u(0) = \alpha_1$.  We arrange $u,v$ decreasingly about their maxima only on $\{i \in \mathbb{Z}\colon i \geq 0\}$ while the construction on the set $\{i \in \mathbb{Z}\colon i \leq 0\}$ is done analogously. If $v(0) =\beta_1$ and both $u,v$ are decreasing about their maxima, there is nothing to prove. Assume this is not the case and that $v(0)  <\beta_1$. Up to reflecting $u$ and $v$, we can assume that there is $m>0$ such that $v(m)  =\beta_1$. Denote by $v(0)=\beta_k<\beta_1$ and $u(m+1)=\alpha_n\leq \alpha_1$ and set $u^1=u_{1,m}$ and $v^1=v_{0,m}$. Using \eqref{ineq:flip2} and Lemma~\ref{lem:minmax} together with $\alpha_1\geq \alpha_n$ and $\beta_1 > \beta_k$, we obtain
\begin{align*}
E^{1\mathrm{d}}_{\mathrm{diag}}(u^1,v^1) = E^{1\mathrm{d}}_{\mathrm{diag}}(u,v) + |\alpha_1-\beta_1|^2 + |\alpha_n-\beta_k|^2 -  |\alpha_n-\beta_1|^2 - |\alpha_1-\beta_k|^2  \leq E^{1\mathrm{d}}_{\mathrm{diag}}(u,v)\,.
\end{align*}  
 Assume now that $(u^{n-1},v^{n-1})$ have been constructed such that for all $1\leq k \leq n-1$ we have $u^{n-1}(k)=\max\{u^{n-1}(i)\colon i\geq k\}$ and  $v^{n-1}(k)=\max\{v^{n-1}(i)\colon i\geq k\}$. The case where $u^{n-1}(k)=\max\{u^{n-1}(i)\colon i\geq k\}$ for all $1\leq k \leq n-1$  and  $v^{n-1}(k)=\max\{v^{n-1}(i)\colon i\geq k\}$ and $1\leq k \leq n-2$ is done analogously. If also $u^{n-1}(n)=\max\{u^{n-1}(i)\colon i\geq n\}$ there is nothing to prove. Assume instead that there is $m >n$ such that $ u^{n-1}(m)=\max\{u^{n-1}(i)\colon i\geq n\}$. Set $u^n= (u^{n-1})_{n,m}$ and $v^n=(v^{n-1})_{n,m-1}$. Then $u^{n-1}(m) > u^{n-1}(n)$ and $v^{n-1}(n-1) \geq v^{n-1}(m)$ and therefore, using Lemma~\ref{lem:minmax} and \eqref{ineq:flip2}, we obtain
\begin{align*}
E^{1\mathrm{d}}_{\mathrm{diag}}(u^n,v^n) \leq E^{1\mathrm{d}}_{\mathrm{diag}}(u^{n-1},v^{n-1}) \,.
\end{align*}
Setting
\begin{align*}
u_{\mathrm{dec}} = \lim_{n\to \infty} u^n \quad \text{ and } \quad  v_{\mathrm{dec}} =\lim_{n\to \infty} v^n,
\end{align*}
the inequality \eqref{ineq:diagdecreasing} follows by noting that the sequence $E^{1\mathrm{d}}_{\mathrm{diag}}(u^n,v^n)$ converges to $E^{1\mathrm{d}}_{\mathrm{diag}}(u_{\mathrm{dec}}, v_{\mathrm{dec}})$. \\ 
\noindent {\bf Step 3:} {\it Proof of \eqref{ineqlem:diagonalsymmetrization}:} We prove the claim by induction. Thanks to Step 2, we can assume that $u,v$ are centered in $0$ and decreasing about their maxima. Without loss of generality we suppose that for some $n \in \mathbb{N}$
\begin{align*}
u(i)=u^*(i)   \quad \text{ for all } -n \leq i \leq n+1 \quad \text{ and } \quad v(i)= \mathrm{R}v^*(i)  \quad \text{ for all } -n \leq i \leq n.
\end{align*}
We then construct $u^n,v^n$ such that 
\begin{align*}
 u^n(i)=u^*(i)  \quad \text{ for all } -n \leq i \leq n+1 \quad \text{ and } \quad v^n(i)= \mathrm{R}v^*(i)  \quad \text{ for all } -n-1 \leq i \leq n
\end{align*}
and 
\begin{align*}
E^{1\mathrm{d}}_{\mathrm{diag}}(u^n,v^n) \leq E^{1\mathrm{d}}_{\mathrm{diag}}(u,v) \,.
\end{align*}
If $v(-n-1)= \mathrm{R}v^*(-n-1)$ there is nothing to prove. Assume that this were not the case, i.e., since  $v$ is decreasing about its maximum in $0$, assume that $v(n+1)= \beta_{2n+2}$ and $v(-n-1)=\beta_k \leq \beta_{2n+2}$. By our assumptions we have $u(n+1) =u^*(n+1)=\alpha_{2n+2}$ and $u(-n) =u^*(-n)=\alpha_{2n+1}$.  Setting now $u^n=u_{-n,n+1}$, $v^n=v_{-n,n}$, noting that $\alpha_{2n+2} \leq \alpha_{2n+1}$, $\beta_k \leq \beta_{2n+2}$ and using Lemma~\ref{lem:minmax}, we obtain
\begin{align*}
E^{1\mathrm{d}}_{\mathrm{diag}}(u^n,v^n)= E^{1\mathrm{d}}_{\mathrm{diag}}(u,v) &- |\alpha_{2n+2}-\beta_{2n+2}|^2 - |\alpha_{2n+1}-\beta_{k}|^2 \\&+ |\alpha_{2n+1}-\beta_{2n+2}|^2 + |\alpha_{2n+2}-\beta_{k}|^2 \leq  E^{1\mathrm{d}}_{\mathrm{diag}}(u,v)\,.
\end{align*}
This concludes the proof.
\end{proof}

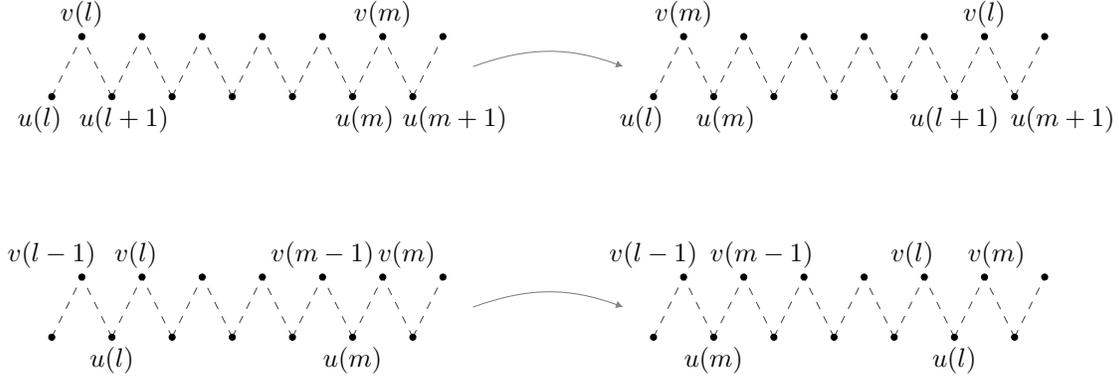
\begin{figure}[htp]
\begin{tikzpicture}[scale=.8]

\tikzset{>={Latex[width=1mm,length=1mm]}};

\draw[dashed,ultra thin,gray!50!black](6,0)--(6.5,1);

\foreach \j in {0,...,5}{
\draw[dashed,ultra thin,gray!50!black](\j,0)--(\j+.5,1)--(\j+1,0);
}

\foreach \j in {0,...,6}{
\draw[fill=black](\j,0) circle(.05);
\draw[fill=black](\j+.5,1) circle(.05);
}

\draw(.5,1) node [anchor=south]{$v(l)$};

\draw(5.5,1) node [anchor=south]{$v(m)$};

\draw(1.2,0) node [anchor=north]{$u(l+1)$};

\draw(5.2,0) node [anchor=north]{$u(m)$};

\draw(-.2,0) node [anchor=north]{$u(l)$};

\draw(6.7,0) node [anchor=north]{$u(m+1)$};

\draw[->,gray](7,.5) to[out=20, in=160] (9.5,.5);


\begin{scope}[shift={(10,0)}]

\draw[dashed,ultra thin,gray!50!black](6,0)--(6.5,1);

\foreach \j in {0,...,5}{
\draw[dashed,ultra thin,gray!50!black](\j,0)--(\j+.5,1)--(\j+1,0);
}

\foreach \j in {0,...,6}{
\draw[fill=black](\j,0) circle(.05);
\draw[fill=black](\j+.5,1) circle(.05);
}

\draw(.5,1) node [anchor=south]{$v(m)$};

\draw(5.5,1) node [anchor=south]{$v(l)$};

\draw(1.2,0) node [anchor=north]{$u(m)$};

\draw(5,0) node [anchor=north]{$u(l+1)$};

\draw(-.2,0) node [anchor=north]{$u(l)$};

\draw(6.8,0) node [anchor=north]{$u(m+1)$};

\end{scope}


\begin{scope}[shift={(0,-4)}]

\draw[->,gray](7,.5) to[out=20, in=160] (9.5,.5);

\draw[dashed,ultra thin,gray!50!black](6,0)--(6.5,1);

\foreach \j in {0,...,5}{
\draw[dashed,ultra thin,gray!50!black](\j,0)--(\j+.5,1)--(\j+1,0);
}

\foreach \j in {0,...,6}{
\draw[fill=black](\j,0) circle(.05);
\draw[fill=black](\j+.5,1) circle(.05);
}

\draw(1.4,1) node [anchor=south]{$v(l)$};

\draw(4.5,1) node [anchor=south]{$v(m-1)$};

\draw(0,1) node [anchor=south]{$v(l-1)$};

\draw(5.9,1) node [anchor=south]{$v(m)$};

\draw(1,0) node [anchor=north]{$u(l)$};

\draw(5,0) node [anchor=north]{$u(m)$};

\end{scope}


\begin{scope}[shift={(10,-4)}]

\draw[dashed,ultra thin,gray!50!black](6,0)--(6.5,1);

\foreach \j in {0,...,5}{
\draw[dashed,ultra thin,gray!50!black](\j,0)--(\j+.5,1)--(\j+1,0);
}

\foreach \j in {0,...,6}{
\draw[fill=black](\j,0) circle(.05);
\draw[fill=black](\j+.5,1) circle(.05);
}

\draw(1.8,1) node [anchor=south]{$v(m-1)$};

\draw(4.3,1) node [anchor=south]{$v(l)$};

\draw(0,1) node [anchor=south]{$v(l-1)$};

\draw(5.7,1) node [anchor=south]{$v(m)$};

\draw(1,0) node [anchor=north]{$u(m)$};

\draw(5,0) node [anchor=north]{$u(l)$};

\end{scope}

\end{tikzpicture}
\caption{The two permutations introduced in Step~1. The dashed lines indicate the interactions}
\label{Fig:diagmin}
\end{figure}

\begin{definition} \label{def:direction-convex} Let $X \subset \mathbb{Z}^d$ and $j \in \{1,\ldots,d\}$. We say that $X$ is convex in direction $e_j$ provided that for all $\alpha \in \Pi_{e_j}$ the set 
\begin{align*}
X^{\alpha,e_j} = \{ t\in \mathbb{Z} \colon \alpha+te_j \in X\}
\end{align*}
is connected (in $\mathbb{Z}$).
\end{definition}

The following theorem shows that the Dirichlet energy  is decreasing under the rearrangement given in Definition~\ref{def:symmetrization-d}.
\begin{theorem} \label{thm:Zdsymdecreases}
    Let $\xi \in \mathcal{B}$ and $u\in\ell^2(\mathbb{Z}^d)$ with $u\geq 0$. Then
    \begin{align*}
    E(u^{*\xi})\leq E(u)\,.
\end{align*} 
If equality holds and if $\xi = e_j$ for some $j\in\{1,\ldots,d\}$ then the following statements are true:
\begin{itemize}
\item[(i)]  For all $t \in \mathbb{R}$ the set $\{u\geq t\}$ (resp. $\{u >t\}$) is convex in direction $e_j$.
\item[(ii)] For all $\alpha \in \Pi_{e_j}$ and all $k \in \{1,\ldots,d\}$, $k\neq j$ there exists a bijection
$
i \colon \mathbb{N} \to \mathbb{Z}
$ 
(depending on $\alpha$ and $k$) such that
\begin{align*}
u^{\alpha+e_k,e_j} (i(n)) \le u^{\alpha+e_k,e_j}(i(n-1)) 
\quad\text{ and }\quad
u^{\alpha,e_j}(i(n)) \le u^{\alpha,e_j}(i(n-1))
\qquad \text{for all } n \ge 1\,.
\end{align*}
In particular, for all $t\geq 0$ 
\begin{align*}
\{u^{\alpha+e_k,e_j}\geq t\} \subset \{u^{\alpha,e_j} \geq t\} \quad \text{ or } \quad \{u^{\alpha,e_j}\geq t\} \subset \{u^{\alpha+e_k,e_j} \geq t\}\,.
\end{align*}
\end{itemize}
\end{theorem}
\begin{proof} Fix $u \in \ell^2(\mathbb{Z}^d)$ such that $u \geq 0$. We distinguish the case $\xi = e_j$ for some $j\in \{1,\ldots,d\}$ and the case $\xi = e_j \pm e_k$ for some $j,k \in \{1,\ldots,d\}$ with $j\neq k$. \\ 
\noindent {\bf Step 1:} {\it $\xi=e_j$ for some $j\in \{1,\ldots,d\}$:} By Lemma~\ref{lem:decomposition}~(1), Lemma~\ref{lem:single slice}, Lemma~\ref{lem:neighbouring-lines}, Definition~\ref{def:symdirections}, and Definition~\ref{def:symmetrization-d},  we have
\begin{align*}
E(u) &= \sum_{\alpha \in \Pi_{e_j}} E^{1\mathrm{d}}(u^{\alpha,e_j}) +  \underset{k\neq j}{\sum_{j=1}^d} \sum_{\alpha \in \Pi_{e_j}} E^{1\mathrm{d}}(u^{\alpha,e_j},u^{\alpha+e_k,e_j})  \\& \geq \sum_{\alpha \in \Pi_{e_j}} E^{1\mathrm{d}}((u^{\alpha,e_j})^*) +  \underset{k\neq j}{\sum_{k=1}^d} \sum_{\alpha \in \Pi_{e_j}} E^{1\mathrm{d}}((u^{\alpha,e_j})^*,(u^{\alpha+e_k,e_j})^*)  \\&=E(u^{*e_j})\,.
\end{align*}
By Lemma~\ref{lem:single slice} and Lemma~\ref{lem:neighbouring-lines},  for all $\alpha \in \Pi_{e_j}$ we have
\begin{align*}
E^{1\mathrm{d}}(u^{\alpha,e_j}) \geq E^{1\mathrm{d}}((u^{\alpha,e_j})^*) \quad \text{ and } \quad E^{1\mathrm{d}}(u^{\alpha,e_j},u^{\alpha+e_k,e_j}) \geq E^{1\mathrm{d}}((u^{\alpha,e_j})^*,(u^{\alpha+e_k,e_j})^*)
\end{align*}
and equality holds only if  $E^{1\mathrm{d}}(u^{\alpha,e_j}) = E^{1\mathrm{d}}((u^{\alpha,e_j})^*) $ for all $\alpha \in \Pi_{e_j}$. Thus Lemma~\ref{lem:single slice} implies that for all $\alpha \in \Pi_{e_j}$ the function $u^{\alpha,e_j}$ is decreasing about its maximum. This implies (i), i.e., that the sets $\{u\geq t\}$ and $\{u> t\}$ are convex in direction $e_j$. Additionally, the equality implies that for all $k \in \{1,\ldots,d\}$, $k \neq e_j$ and all $\alpha \in \Pi_{e_j}$ we have
\begin{align*}
E^{1\mathrm{d}}(u^{\alpha,e_j},u^{\alpha+e_k,e_j})  =E^{1\mathrm{d}}((u^{\alpha,e_j})^*,(u^{\alpha+e_k,e_j})^*) \,.
\end{align*}
By Lemma~\ref{lem:neighbouring-lines} the latter equality implies \eqref{eq:monotone-reordering}, hence (ii). This concludes the proof in the first case. \\
\noindent {\bf Step 2:} {\it $\xi=e_j \pm e_l$ for some $j,l\in \{1,\ldots,d\}$ with $j\neq l $:} We assume $\xi=e_j + e_l$.  By Lemma~\ref{lem:decomposition}~(2), Lemma~\ref{lem:neighbouring-lines}, Lemma~\ref{lem:diagonalsym}, Definition~\ref{def:symdirections}, and Definition~\ref{def:symmetrization-d},  we have
\begin{align*}
E(u) &= \sum_{\alpha \in \Pi_\xi^0} \left(E^{1\mathrm{d}}_{\mathrm{diag}}(u^{\alpha,\xi},u^{\alpha+e_j,\xi})+E^{1\mathrm{d}}_{\mathrm{diag}}(u^{\alpha,\xi},u^{\alpha+e_l,\xi}) \right) \\&\quad+\underset{k\notin \{j,l\}}{\sum_{k=1}^d} \left(  E^{1\mathrm{d}}(u^{\alpha,\xi},u^{\alpha+e_k,\xi})+E^{1\mathrm{d}}(u^{\alpha+e_j,\xi},u^{\alpha+e_j+e_k,\xi}) \right) \\& \geq \sum_{\alpha \in \Pi_\xi^0} \left( E^{1\mathrm{d}}_{\mathrm{diag}}((u^{\alpha,\xi})^*,\mathrm{R}(u^{\alpha+e_j,\xi})^*) +E^{1\mathrm{d}}_{\mathrm{diag}}((u^{\alpha,\xi})^*,\mathrm{R}(u^{\alpha + e_l,\xi})^*)  \right) \\&\quad + \sum_{\alpha \in \Pi_{\xi}^0}\underset{k\notin \{j,l\}}{\sum_{k=1}^d} \left( E^{1\mathrm{d}}((u^{\alpha,\xi})^*,(u^{\alpha+e_k,\xi})^*) +E^{1\mathrm{d}}(\mathrm{R}(u^{\alpha+e_j,\xi})^*,\mathrm{R}(u^{\alpha+e_j+e_k,\xi})^*) \right) 
\\& = E(u^{*\xi})\,.
\end{align*}
This concludes the proof in the second case. 
\end{proof}

\begin{definition}(The $\mathcal{P}_n$-property)\label{def:Pn-property}
    Let $n\in \mathbb N$. A function $u\in \ell^2(\mathbb{Z}^d)$ satisfies $\mathcal{P}_n$ if  for all  $\xi_1,\cdots,\xi_n\in \mathcal{B}$ there holds
    \begin{align*}
    E\big(u^{*(\xi_1,\cdots,\xi_n)}\big)=E(u)\,.
    \end{align*}
\end{definition}
\begin{remark}\label{rem:Pn-propertysym} Clearly if $u\in \ell^2(\mathbb{Z}^d)$ satisfies the property $\mathcal{P}_n$ for some $n \in \mathbb{N}$ then for all $\xi \in \mathcal{B}$ the function $u^{*\xi}$ satisfies the property $\mathcal{P}_{n-1}$.
\end{remark} 

\begin{proposition} \label{cor:FK-min}
     Let $X\subset \Zd$ with $\#X=N$  be such that $\lambda_N(X)=m_{\lambda,N}$ and let $u\colon \mathbb{Z}^d \to \mathbb{R}$ be its equilibrium potential.  Then $u$ satisfies $\mathcal{P}_n$ for all $n \in \mathbb{N}$ and $\{u>0\}=X$. In particular the set $X$ is convex in direction $e_j$ for all $ j\in \{1,\ldots,d\}$. 
\end{proposition}
\begin{proof} Let $u \in \ell^2(\mathbb{Z}^d)$ be such that 
\begin{align*}
E_2(u) = \lambda_N(X)=m_{\lambda,N}\,.
\end{align*}
 First, we remark that, due to  \cite[Proposition~3.4, Proposition~3.5]{CiKrLeMo25}, $X$ is connected and $X = \{u>0\}$. Fix $\xi \in \mathcal{B}$ and set $X^*=\{u^{*\xi} >0\}$. Then, by Definition~\ref{def:symmetrization-d}, we have $\#X^*=\#\{u^{*\xi} >0\}=\#\{u >0\}=\#X =N$, and therefore, by Theorem~\ref{thm:Zdsymdecreases}, 
 \begin{align*}
 m_{\lambda,N} =E_2(u) \geq E_2(u^{*\xi}) \geq  \lambda_N(X^*) \geq  m_{\lambda,N}\,.
 \end{align*}
Thus $E_2(u) = E_2(u^{*\xi})$. By the arbitrariness of $\xi \in \mathcal{B}$, we have that $u$ satisfies $\mathcal{P}_1$. Iterating the argument and using Remark \ref{rem:Pn-propertysym} we infer that $u$ satisfies $\mathcal{P}_n$ for all $n\in \mathbb{N}$. In particular the validity of $\mathcal{P}_1$ for $\xi =e_j$ with $j \in \{1,\ldots,d\}$ together with  Theorem~\ref{thm:Zdsymdecreases}, implies that set $X$ is convex in direction $e_j$ for all $j \in \{1,\ldots,d\}$. 
\end{proof}

\begin{proposition} \label{prop X is superlevelset}
    Let $X\subset\Zd$ be such that $\#X=N$. 
\begin{itemize}
\item[(i)] Let $u \colon \mathbb{Z}^d \to \mathbb{R}$ be the capacitary potential of $X$ and
\begin{align*}
E(u) = \mathrm{cap}_{N}(X) = m_{N}\,.
\end{align*}
 Then $u$ satisfies $\mathcal{P}_n$ for all $n \in \mathbb{N}$ and $\{u\geq 1\}=\{u=1\}=X$. In particular $X$ is convex in direction $e_j$ for all $j\in \{1,\ldots,d\}$.
\item[(ii)] Fix $R>2$ and let $u \colon \mathbb{Z}^d \to \mathbb{R}$ be the relative capacitary potential of $X$ and
\begin{align*}
E(u) = \mathrm{cap}_{N}(X,R) = m_{N}(R)\,.
\end{align*}
 Then $u$ satisfies $\mathcal{P}_n$ for all $n \in \mathbb{N}$ and  $\{u\geq 1\}=\{u=1\}=X$. In particular $X$ is convex in direction $e_j$ for all $j\in \{1,\ldots,d\}$.
\end{itemize}    
\end{proposition}
\begin{proof} We prove only (i), the proof of (ii) being analogous. We first show that
\begin{align}\label{eq:ueq1X}
\{u\geq 1\}=\{u=1\}=X\,.
\end{align} 
Let $X \subset \mathbb{Z}^d$ with $\#X=N$ be an optimal set for $m_{N}$ and let $u\colon \mathbb{Z}^d\to \mathbb{R}$ be its capacitary potential, i.e.,
\begin{align}\label{eqproof:pcapacitypotential}
E(u) = \mathrm{cap}_{N}(X) = m_{N}\,.
\end{align} 
By Proposition~\ref{prop:function energie min} (1),(2) we have that $X \subset \{u=1\} =\{u\geq 1\}$. 
Assume by contradiction that the inclusion is strict, i.e., there exists $y_0 \in  \{u=1\} \setminus X$. Fix $x_0 \in X$ such that $(x_0)_1 \geq (x)_1$ for all $x \in X$ and define $\hat{X} = X \cup\{y_0\} \setminus \{x_0\}$. As $u(x_0)=1$ and  $(x_0)_1 \geq (x)_1$ for all $x\in X$ we have that $u$ is not harmonic at $x_0$, hence it is not the capacitary potential of $\hat{X}$. Denote the latter by $\hat u$ (its existence and uniqueness follow from Proposition~\ref{prop:cap_pot}). We have that 
 \begin{align*}
 E(\hat{u}) < E(u)\,.
 \end{align*}
 Thus
 \begin{align*}
 m_{N} \leq \mathrm{cap}_{N}(\hat{X}) =  E(\hat{u}) < E(u)   = \mathrm{cap}_{N}(X) = m_{N}\,.
 \end{align*}
 This is a contradiction and shows \eqref{eq:ueq1X}.  
 Fix now $\xi \in \mathcal{B}$ and consider $u^{*\xi}$ together with $X^*=\{u^{*\xi}=1\} $ . Then, by Definition~\ref{def:symmetrization-d} and \eqref{eq:ueq1X}, we have
 \begin{align*}
 \#X^* =\#\{u^{*\xi}=1\} = \#\{u=1\} = N  
\end{align*}  
 and therefore, by Theorem~\ref{thm:Zdsymdecreases}, we obtain
 \begin{align*}
  m_{N} \leq \mathrm{cap}_{N}(X^*) \leq E(u^{*\xi}) \leq   E(u) = \mathrm{cap}_{N}(X)= m_{N} \,.
 \end{align*}
The latter chain of inequalities implies $E(u^{*\xi}) = E(u) $ which, by the arbitrariness of $\xi \in \mathcal{B}$, ensures that $u$ satisfies $\mathcal{P}_n$ for all $n \in \mathbb{N}$. Finally, property $\mathcal{P}_1$ applied for $\xi =e_j$ for $j \in \{1,\ldots,d\}$ together with  Theorem~\ref{thm:Zdsymdecreases} let us conclude that $X$ is convex in direction $e_j$ for all $j \in \{1,\ldots,d\}$. 
\end{proof}

\subsection{Diameter estimate of optimal sets} \label{sec:diameter estimate}
The main goal of this section is to prove an asymptotic diameter bound (with uniform constant) for optimal sets in the isocapacitary problem and for minimizers of the first Dirichlet eigenvalue of the combinatorial Laplacian on $\Zd$. 
To this end we recall that the first Dirichlet eigenvalue of the combinatorial Laplacian of a set $X \subset \mathbb{Z}^d$ with $\#X =N$ is defined as
\begin{align*}
\lambda_N(X)=\min\Big\{E_2(u)\,\colon\, u(i)=0 \text{ on }\Zd\setminus X\,,\, \frac{1}{N}\sum_{i\in\Zd}|u(i)|^2=1 \Big\} \,.
\end{align*}
We say that $X$ is an optimal set if
\begin{align*}
\lambda_N(X) = m_{\lambda,N}\,, \quad \text{ where} \quad m_{\lambda,N}&=\inf_{\substack{X\subset \Zd\\\#X=N}} \lambda_N(X)\,.
\end{align*}
\begin{proposition}\label{prop:diameterestimate} Let $p>1$ and let $X \subset \mathbb{Z}^d$ with $\#X=N$. There exists $C_{d}>0$ such that
 \begin{itemize}
 \item[(i)] \label{ineqprop:p-capacityperimeter} capacity: if $\cappn{X}= m_{N}$, then $\mathrm{diam}(X) \leq C_{d}N^{1/d}$,\\
     \item[(ii)]\label{ineq:alphamin-R} relative capacity: if $\caprn{X}= m_{N}(R)$, then $\mathrm{diam}(X) \leq C_{d}N^{1/d}$,\\
      \item[(iii)]\label{ineq:Eigenvalue} first eigenvalue of the Laplacian: if $\lambda_N(X)= m_{\lambda,N}$, then $\mathrm{diam}(X) \leq C_{d}N^{1/d}$
\end{itemize}
\end{proposition}
We postpone the proof of Proposition~\ref{prop:diameterestimate} to the end of this section and instead show here that it implies Proposition~\ref{prop:perimeterestimate}.

\begin{proof}[Proof of Proposition~\ref{prop:perimeterestimate}] We prove only (1), as the proof of (2) is analogous. We first observe that
\begin{align}\label{eq:slicing-of-perimeter}
\begin{split}
\mathrm{P}(X) &= \sum_{i \in X} \#\{ j \in \mathbb{Z}^d\setminus X \colon |i-j|=1\} = \sum_{k=1}^d  \sum_{i \in \mathbb{Z}^d} |\mathrm{1}_X(i+e_k) - \mathrm{1}_{\mathbb{Z}^d\setminus X}(i)| \\&=  \sum_{k=1}^d \sum_{\alpha \in \Pi_{e_k}} \sum_{t \in \mathbb{Z}} |\mathrm{1}_X(\alpha+(t+1)e_k) - \mathrm{1}_{\mathbb{Z}^d\setminus X}(\alpha +te_k)|\,.
\end{split}
\end{align}
Thanks to Proposition~\ref{prop X is superlevelset}~(1), the set $X$ is convex in direction $e_j$ for all $j\in\{1,2,\dots,d\}$. Therefore, for all $\alpha \in \Pi_{e_k}$ we have
\begin{align}\label{eq:convexsetsperimeterslice}
\sum_{t \in \mathbb{Z}} |\mathrm{1}_X(\alpha+(t+1)e_k) - \mathrm{1}_{\mathbb{Z}^d\setminus X}(\alpha +te_k)| = \begin{cases} 2 &\text{if } X^{\alpha,e_k} \neq \emptyset\,,\\
0 &\text{otherwise.}
\end{cases}
\end{align}
Finally, observe that (by projecting $X$ onto $\Pi_{e_k}$) by Proposition~\ref{prop:diameterestimate} ({\textit i}) we have
\begin{align*}
\#\{ \alpha \in \Pi_{e_k} \colon X^{\alpha,e_k} \neq \emptyset\} \leq \mathrm{diam}(X)^{d-1} \leq C_dN^{\frac{d-1}{d}}\,.
\end{align*}
This together with \eqref{eq:slicing-of-perimeter}, \eqref{eq:convexsetsperimeterslice}, and the definition of $\mathrm{P}_N$ implies the claim.
\end{proof}
Given $x_0 \in \mathbb{Z}^d$ and $I_k = \{i_1,\ldots,i_k\} \subset \{1,\ldots,d\}$ we define
\begin{align*}
\Pi(x_0,I_k) =  x_0 + \mathrm{span}_{\mathbb{Z}}\{e_{i_1},\ldots,e_{i_k}\} \,.
\end{align*}
Additionally, we set $S(\alpha,e_j)= \{\alpha+te_j \colon t \in \mathbb{Z}\}$.
\begin{theorem} \label{thm:propertyPn-walled-in}
    Let $u\in \ell^2(\mathbb{Z}^d)$ be such that $u\geq 0$ and $\mathcal{P}_{n}$ holds for all $n\in \mathbb{N}$. Fix $t \geq 0$ and $j \in \{1,\ldots,d\}$. The following two statements hold true:
    \begin{itemize}
    \item[(i)]   For all $x_0 \in \mathbb{Z}^d$ and all $I_k= \{i_1,\ldots,i_k\}  \subset \{1,\ldots,d\}$ with $j \in I_k$ there exists $\alpha \in \Pi_{e_j}$ such that $S(\alpha,e_j) \subset \Pi(x_0,I_k) $ and
    \begin{align*}
S(\beta,e_j)\subset \Pi(x_0,I_k) \quad \implies \quad\{u^{\beta,e_j} \geq t\}  \subset \{u^{\alpha,e_j} \geq t\}\quad \text{ for all }\beta \in \Pi_{e_j}\,.
\end{align*} 
\item[(ii)] For all $x_0 \in \mathbb{Z}^d$ and all $ I_k= \{i_1,\ldots,i_k\}  \subset \{1,\ldots,d\}$ with $j \in I_k$ there exists $\alpha \in \Pi_{e_j}$ such that $S(\alpha,e_j) \subset \Pi(x_0,I_k) $ and
    \begin{align*}
S(\beta,e_j)\subset \Pi(x_0,I_k) \quad \implies \quad\{u^{\beta,e_j} > t\}  \subset \{u^{\alpha,e_j} > t\}\quad \text{ for all }\beta \in \Pi_{e_j}\,.
\end{align*} 
    \end{itemize}
\end{theorem}
Note that the statement (ii) of the theorem above is a generalization of the {\it walled-in property}, see \cite{Z2}[Definition 6.3].
\begin{proof}[Proof of Theorem~\ref{thm:propertyPn-walled-in}] We prove only {\rm (i)} as the proof of {\rm (ii)} is analogous. In this case we can assume that $t>0$ since otherwise there is nothing to prove. We first set up some notation: we write $X=\{u\geq t\}$ and $X^{\alpha,e_j} =\{u^{\alpha,e_j} \geq t\}$.  We prove the claim by induction.  We show that for all $u$ satisfying $\mathcal{P}_n$, all $x_0 \in \mathbb{Z}^d$, $k \in \{1,\ldots,d\}$, $ I_k=\{i_1,\ldots,i_k\} \subset \{1,\ldots,d\}$, and all $\alpha, \beta \in \mathbb{Z}^d$ the following statement is true: 
\begin{align}\label{induction:goal}
S(\alpha,e_j),S(\beta,e_j)\subset \Pi(x_0,I_k) \quad \implies \quad  X^{\alpha,e_j} \subset  X^{\beta,e_j}  \quad  \text{ or } \quad X^{\beta,e_j} \subset  X^{\alpha,e_j}\,.
\end{align}
 This implies that the sets $X^{\alpha,e_j}, \alpha \in \mathbb{Z}^d$ such that $S(\alpha,e_j) \subset \Pi(x_0,I_k)$  are ordered with respect to set inclusion. As $\#X <+\infty$ there exists $\alpha \in \mathbb{Z}^d$ for which $X^{\alpha,e_j}$ is maximal  with respect to set inclusion. This implies the statement of the Theorem. We now turn to the proof of \eqref{induction:goal}. \\
{\noindent \it Base case: $k=2$.} We assume, without loss of generality, $x_0=0$ and $I_2=\{1,2\}$ with $j=1$. In this case the relevant $\alpha$ are of the form $\alpha = k e_2$ with $k \in \mathbb{Z}$. We denote by $N_k =\# X^{ke_2,e_1}$ and let $m_0 \in \mathbb{Z}$ be such that $N_{m_0} \geq N_m$ for all $m\in \mathbb{Z}$.  We claim that
\begin{align}\label{induction:keq2}
X^{(\pm (m+1)+m_0)e_2,e_1}  \subset X^{(m+m_0)e_2,e_1}    \text{ for all } m \in \mathbb{N}\,.
\end{align}
This clearly implies \eqref{induction:goal} for $k=2$ under our assumptions on $I_k$ and $x_0$. 
We now prove \eqref{induction:keq2}. To this end we assume for simplicity that $N_0 \geq N_m$ for all $m \in \mathbb{Z}$ and prove the claim in this case. The argument being the same in the general case, we only show the proof for $m\geq 0$. We proceed by induction. Recalling that $u$ satisfies $\mathcal{P}_n$, by Theorem~\ref{thm:Zdsymdecreases}(ii) written for $j=1$, $k=2$ and $\alpha=0$ we have $X^{e_2,e_1} \subset X^{0,e_1}$ (the other inclusion cannot hold true due to our assumption on $N_0$). 
Now assume that   $X^{ne_2,e_1} \subset X^{(n-1)e_2,e_1}$ for all $1\leq n \leq m$ and we prove that $X^{(m+1)e_2,e_1} \subset X^{me_2,e_1}$. 
Assume by contradiction that this were not true, i.e. $ X^{me_2,e_1} \subsetneq X^{(m+1)e_2,e_1} $, i.e., in particular, again by Theorem~\ref{thm:Zdsymdecreases}(ii) applied with $j=1$, $k=2$ and $\alpha=me_2$, we have  $N_0 \geq N_{m+1} >N_m$. 
As a consequence of that, observing that $\{(u^{*e_1})^{ke_2,e_1}\geq t\} = \left[-\left\lfloor\frac{N_k-1}{2}\right\rfloor, \left\lceil\frac{N_k-1}{2}\right\rceil\right]\cap \mathbb{Z}$, we deduce that $u^{*e_1}$ is not convex in direction $e_2$. Thanks to Theorem~\ref{thm:Zdsymdecreases}(i), this implies that
\begin{align*}
E(u^{*(e_1,e_2)}) = E((u^{*e_1})^{*e_2})  <  E(u^{*e_1}) \leq E(u)\,.
\end{align*}
This shows that $u$ does not satisfy $\mathcal{P}_2$, which si a contradiction and proves \eqref{induction:keq2}.  \\
{\noindent \it Induction step: $k\to k+1$.} We assume that \eqref{induction:goal} is satisfied for all $k$ with $1\leq k <d$. Our aim is to prove it for $k+1$. We assume, without loss of generality, $x_0=0$ and $I_{k+1}=\{1,\ldots,k+1\}$ with $j=1$. Assume by contradiction that \eqref{induction:goal} were not true, i.e., there exists $\alpha \in \mathbb{Z}^d$, and $\beta= \sum_{n=1}^{k+1}\lambda_n e_n \in \mathbb{Z}^d$ with $\lambda_n \in \mathbb{Z}\setminus \{0\}$ for all $n=1,\ldots,k+1$ (this is necessary as otherwise \eqref{induction:goal} would be applicable for some $I_k$)  such that $X^{\alpha,e_1} \setminus  X^{\beta,e_1} \neq \emptyset$ and $X^{\beta,e_1} \setminus  X^{\alpha,e_1} \neq \emptyset$. By the induction assumption we can assume that
\begin{align}\label{incl:alpha}
 &X^{\hat\alpha,e_1} \subset  X^{\alpha,e_1}  \text{ for all } \hat{\alpha} = \alpha +\sum_{n=1}^k \mu_n e_n\,, \mu_n \in \mathbb{Z}
\end{align}
and 
\begin{align}\label{incl:beta}
X^{\hat\beta,e_1} \subset  X^{\beta,e_1}  \text{ for all } \hat{\beta} = \beta +\sum_{n=1}^k \mu_n e_n\,, \mu_n \in \mathbb{Z} \,.
\end{align}
We first note that also $u^{*e_{k+1}}$ satisfies $\mathcal{P}_n$. For $x_0 \in \mathbb{Z}^d$ we set $X_*^{x_0,e_1}= \{(u^{*e_{k+1}})^{x_0,e_1}\geq t\}$. 
For $\alpha_0 = \alpha - \alpha_{k+1} e_{k+1}$,  $\beta_0= \beta - \beta_{k+1} e_{k+1}$, due to \eqref{incl:alpha} and \eqref{incl:beta}, we have
\begin{align*}
X_*^{\alpha_0,e_1} = X^{\alpha,e_1} \quad \text{ and } \quad X_*^{\beta_0,e_1} = X^{\beta,e_1}  
\end{align*}
Thus we obtain
\begin{align*}
X_*^{\alpha_0,e_1} \setminus X_*^{\beta_0,e_1}= X^{\alpha,e_1} \setminus X^{\beta,e_1} \neq \emptyset \quad \text{ and } \quad X_*^{\beta_0,e_1} \setminus X_*^{\alpha_0,e_1}  = X^{\beta,e_1} \setminus  X^{\alpha,e_1}\neq \emptyset\,.
\end{align*}
This is a contradiction to the induction assumption applied to $u^{*e_{k+1}}$  with $I_k=\{1,\ldots,k\}$ and $\alpha_0$ and it concludes the proof.
\end{proof}
\begin{proof}[Proof of Proposition~\ref{prop:diameterestimate}] It suffices to prove the claim for $u \in \ell^2(\mathbb{Z}^d)$, $u\geq 0$ satisfying the condition $\mathcal{P}_n$ for all $n \in \mathbb{N}$. We show the proof only in case (i) as the other two cases are analogous. 
We set $X=\{u \geq 1\}$ and $X^{\alpha, e_j}= \{u^{\alpha,e_j} \geq 1\}$. 
We show that there exist $\xi_1,\ldots, \xi_{3d-2} \in \mathcal{B}$ and a dimensional constant $C_d >0$ such that for $u^{*(\xi_1,\ldots,\xi_{3d-2})}$ we have 
\begin{align}\label{ineq:diameter-volume}
\#\{u^{*(\xi_1,\ldots,\xi_{3d-2})} \geq 1\} \geq C_d\,\mathrm{diam}(X)^d\,.
\end{align}
Due to Proposition~\ref{prop X is superlevelset}(i) we have that $\#X=N$ and, as $u^{*\xi}$ is only permuting values of $u$, we thus have that $N=\#X =\#\{u^{*(\xi_1,\ldots,\xi_{3d-2})} \geq 1\}$. This together with \eqref{ineq:diameter-volume} implies the claim of the Proposition. We are left to prove \eqref{ineq:diameter-volume}. In the following, given $X= \{w \geq t\}$ we write $X^{*\xi} = \{w^{*\xi}\geq 1\}$. \\
\noindent {\bf Step 1:} {\it $X$ contains a line of large length:} We show that (up to rotating $X$) there exists $\alpha \in \mathbb{Z}^d$ such that 
\begin{align}\label{ineq:sliceX}
[t_1,t_2] \cap \mathbb{Z} \subset X^{\alpha,e_1}\quad \text{ with } \quad \sqrt{d}(t_{2}-t_{1})\geq  \mathrm{diam}(X)\,.
\end{align}
 By definition we have that there exist $x_1,x_2 \in X$ such that
$
\mathrm{diam}(X) = |x_1-x_2|
$. Without loss of generality we can assume, up to rotation (and possibly exchanging $x_1$ and $x_2$),  that $(x_1-x_2)_1 \geq |(x_1-x_2)_j|$ for all $j=1,\ldots,d$. Thus 
\begin{align*}
\mathrm{diam}(X) =  |x_1-x_2| \leq \sqrt{d}(x_1-x_2)_1 \,.
\end{align*}
Now, Theorem~\ref{thm:propertyPn-walled-in}(i) applied for $k=d$ and $j=1$ implies that there exists $\alpha \in \Pi_{e_1}$ such that $\{(x_{1})_1,(x_{2})_1\} \subset X^{\alpha,e_1}$. Due to Proposition~\ref{prop X is superlevelset}(i) the set $X$ is direction convex in direction $e_1$ and therefore $[(x_{1})_1,(x_{2})_1] \cap \mathbb{Z} \subset X^{\alpha,e_1}$. Setting $t_1=(x_1)_1$ and $t_2=(x_2)_1$, this implies \eqref{ineq:sliceX}. \\
\noindent {\bf Step 2:} {\it Iterated symmetrization:} Before we state the precise statement of this step we introduce some notation: Given $m \in \mathbb{N}$, $k,j \in \{1,\ldots,d\}$ and $x \in \mathbb{Z}^d$ we set
\begin{align*}
[0,m\, e_k] = \{l e_k \colon l \in \mathbb{N}\,, 0\leq l \leq m\} \text{ and } [0,m\, e_k] \times [0,m\, e_j] = \{\lambda e_k  +\mu e_j \colon    \lambda,\mu \in \mathbb{N}\,, 0\leq \lambda,\mu \leq m\}\,.
\end{align*}
Similarly, we define finite products of such intervals. Additionally, we write $[x,x+m\, e_k] = x+ [0,m\, e_k] $.
 In this step we prove the existence of $c>0$ (for simplicity we assume that $c$ and $m$ are such that a $c^km\in\N$, for all $k\in\{1,\dots,d\}$) such that the following holds: If 
\begin{align*}
[x,x+m\,e_k] \cap \mathbb{Z}^d \subset X
\end{align*}
for some $x\in \Pi(0,\{e_1,\ldots,e_{k-1}\})$ and $m \in \mathbb{N}$, then $X_{k+1}= X^{*(e_k+e_{k+1},e_{k+1},e_k)}$ satisfies\\
\begin{align*}
x+[0,c m\, e_k] \times [0, cm\, e_{k+1}] \cap \mathbb{Z}^d \subset X_{k+1}\,.
\end{align*}
 We postpone the proof of this claim and show first how it allows us to conclude. \\
 \noindent {\bf Step 3:} {\it Conclusion:} Due to Step~1 (assuming up to rotation that $j=1$), setting $X_1=X^{*e_1}$, there exists  $m= t_2-t_1$ such that
\begin{align*}
 [0,m\, e_1] \cap \mathbb{Z}^d \subset X_{1} \quad \text{ and } \quad 2\sqrt{d}\,m \geq \mathrm{diam}(X)\,.
\end{align*}
Now, setting $X_{k+1} = (X_k)^{*(e_k+e_{k+1},e_{k+1},e_k)}$, we have
\begin{align}\label{incl:Xkplus1}
\prod_{j=1}^{k+1} [0, c^{k}m e_j] \cap \mathbb{Z}^d \subset X_{k+1}\,.
\end{align}
Indeed, by \eqref{ineq:sliceX} and by induction (for $k=1$ this is  Step~1), we can assume that
\begin{align*}
\bigcup_{x \in \prod_{j=1}^{k-1}[0,c^{k-1}\, m \, e_j] \cap\mathbb{Z}^d}\left( x + [0,c^{k-1} \, m \, e_k] \cap \mathbb{Z}^d\right)=\prod_{j=1}^{k} [0,c^{k-1}\, m\, e_j] \cap \mathbb{Z}^d \subset X_{k}\,.
\end{align*}
Applying Step~2, we obtain
\begin{align*}
\prod_{j=1}^{k+1} [0,c^{k}\, m e_j] \cap \mathbb{Z}^d=\bigcup_{x \in \prod_{j=1}^{k-1}[0,c^{k}\, m \, e_j] \cap\mathbb{Z}^d} \left( x + [0,c^{k}\, m \, e_k] \times [0,c^{k} \, m \, e_{k+1}]\right)\cap \mathbb{Z}^d\subset  X_{k+1}\,.
\end{align*}
Now for $k+1=d$, we obtain
\begin{align*}
N=\#X= \#X_{d} \geq c^{d^2-d} m^d \geq c^{d^2-d} d^{-\frac{d}{2}}2^{-d} \mathrm{diam}(X)^d\,. 
\end{align*}
This concludes the proof of \eqref{ineq:diameter-volume}.\\
\noindent {\it Proof of  Step~2}: We now turn to the proof of Step~2, see Figure~\ref{Fig:Step2} for illustration. Fix $k \in \{1,\ldots,d\}$, $x \in \Pi(0,\{e_1,\ldots,e_{k-1}\})$ ($x=0$ if $k=1$) and $m \in \mathbb{N}$ such that $[x,x+m\, e_k]\cap \mathbb{Z}^d\subset X$. Note that
\begin{align}\label{incl:X*diagonal}
\left\{ y \in \Pi(x,\{e_k,e_{k+1}\}) \cap \Pi_{e_k+e_{k+1}} \colon 0\leq y_k \leq \frac{m}{2}   \,, -\frac{m}{2} \leq y_{k+1} \leq 0  \right\} \subset X^{*(e_k+e_{k+1})}\,,
\end{align}
since for $y$ in the above set we have that $\{y+t(e_k+e_{k+1})\colon t \in \mathbb{Z}\} \cap X \neq \emptyset$. 
We now claim that we can find $\alpha_0 \in \Pi(x,\{e_k, e_{k+1}\})$ such that 
\begin{align}\label{incl:alpha0X*}
\alpha_0 + \left[0, \frac{m}{8} e_k\right] \times  \left[0, \frac{m}{8} e_{k+1}\right] \cap \mathbb{Z}^d \subset X^{*(e_k+e_{k+1})}\,.
\end{align}
If this claim is true, then clearly,
\begin{align*}
x + \left[0, \frac{m}{32} e_k\right] \times  \left[0, \frac{m}{32}e_{k+1}\right] \cap \mathbb{Z}^d \subset X^{*(e_k+e_{k+1},e_{k+1},e_k)}
\end{align*}
and Step~2 follows with $c=\frac{1}{32}$. 
We are left to prove \eqref{incl:alpha0X*}. Using \eqref{incl:X*diagonal}, the fact that $x \in X^{*(e_k+e_{k+1})}$, and applying Theorem~\ref{thm:propertyPn-walled-in}(i) with $x_0=x$, $I_2=\{k,k+1\}$ and $j = e_k$, there exists $l \in \mathbb{Z}$ such that for $\alpha_0= x+le_{k+1}$ we have 
\begin{align}\label{incl:horizontal-line}
\alpha_0 + \left[0,\frac{m}{2} e_k\right] \cap \mathbb{Z}^d \subset X^{*(e_k+e_{k+1})}\,.
\end{align}
By symmetry we can assume that $l \leq -\frac{m}{4}$. As $X^{*(e_k+e_{k+1})}$ is convex in direction $e_{k+1}$, using \eqref{incl:X*diagonal} and \eqref{incl:horizontal-line}, we obtain \eqref{incl:alpha0X*}. This concludes the proof.
\end{proof}
\begin{figure}[htp]
\begin{tikzpicture}

\tikzset{>={Latex[width=1mm,length=1mm]}};

\draw[->](0,0)--++(1,1);

\draw(0,0)++(1,1) node[anchor = south west] {$e_k+e_{k+1}$};

\draw[white,fill=gray!50!white](0,-1.2)--++(.6,0)--++(0,.6)--++(-.6,0)--++(0,-.6);

\draw[->,gray!50!black](-1,0)--++(5,0);

\draw(-1,0)++(5,0) node [anchor=south] {$e_k$};

\draw[->,gray!50!black](0,-3)--++(0,4);

\draw(0,-3)++(0,4) node [anchor=east] {$e_{k+1}$};;

\draw[ultra thick,black](0,0)--++(3,0); 

\draw(0,0) node [anchor= south east] {$x$};

\draw(3,0) node [anchor= south ] {$x+me_k$};

\draw[dashed,ultra thin, gray!80!black](3,0)--(0,-3);

\draw[gray, ultra thick](0,0)--++(1.5,-1.5);

\draw[gray, ultra thick](0,-1.2)--++(1.5,0);

\draw(0,-1.2) node [anchor=east]{$x+le_{k+1}$};

\end{tikzpicture}
\caption{The construction in Step~2: The bold black line is contained in $X$, whereas the two bold gray lines are contained in $X^{*(e_k+e_{k+1})}$. Hence, also the light gray square is contained in $X^{*(e_k+e_{k+1})}$.}
\label{Fig:Step2}
\end{figure}
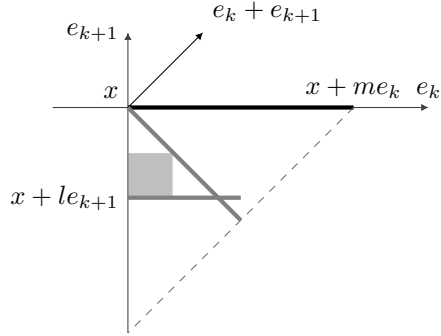

%
%
%
%
%
%

\section*{Acknowledgment}
The research of L.~Kreutz was supported by the DFG through the Emmy Noether Programme (project number 509436910).
I.~Mansoor gratefully acknowledges the hospitality of the Technical University of Munich where this research project was carried out during his internship undertaken as part of the M1 Jacques Hadamard program.

\end{document}